\documentclass[12pt]{article}
\usepackage[english]{babel}
\usepackage{graphicx}
\usepackage{framed}
\usepackage[normalem]{ulem}
\usepackage{amsmath}
\usepackage{amsthm}
\usepackage[utf8]{inputenc}
\usepackage{amssymb}
\usepackage{amsfonts}
\usepackage{tikz-cd}
\usepackage{enumerate}
\usepackage{tikz}
\usepackage[top=1 in,bottom=1in, left=1 in, right=1 in]{geometry}

\usepackage{dynkin-diagrams}
\usetikzlibrary{decorations.markings}

\usepackage{hyperref}

\newcommand{\A}{\mathcal{A}}
\newcommand{\B}{\mathcal{B}}
\newcommand{\C}{\mathcal{C}}

\newcommand{\s}{\mathcal{O}}
\newcommand{\Q}{\mathbb{Q}}
\newcommand{\Z}{\mathbb{Z}}

\renewcommand{\emptyset}{\varnothing}

\theoremstyle{plain}
\newtheorem{theorem}{Theorem}[section] 

\theoremstyle{definition}
\newtheorem{prop}[theorem]{Proposition}

\newtheorem{Remark}[theorem]{Remark}

\newtheorem{cor}[theorem]{Corollary}

\newtheorem{defn}[theorem]{Definition}

\usepackage{mathtools}
\usepackage{multirow}

\tikzstyle{vertex} = [fill,shape=circle,node distance=80pt]
\tikzstyle{edge} = [fill,opacity=.5,fill opacity=.5,line cap=round, line 
join=round, line width=50pt]
\tikzstyle{elabel} =  [fill,shape=circle,node distance=30pt]

\title{Distinction of the Steinberg representation and the dual group of a symmetric space}
\author{Guy Shtotland }
\date{}

\begin{document}

\maketitle
\setlength\parindent{0pt}

\begin{abstract}
We study the distinction of the Steinberg representation of a split reductive 
group $G$ with respect to a split symmetric subgroup $H \subset G$.  We relate 
this distinction problem to a problem about the existence of a non-zero 
harmonic function on a certain hyper-graph related to $X = G/H$. We verify the 
relative local Langlands conjecture for the Steinberg representation by showing 
that over a $p$-adic field the Steinberg representation is $H$-distinguished if 
and only if its Langlands parameter factors through the dual group of $X$.
\end{abstract}

\begin{section}{Introduction}

A fundamental problem in relative representation theory of a reductive group $G$ with 
a symmetric subgroup $H$, is the study of $H$-distinguished $G$ representations.
 A representation $\pi$ of $G$ is called $H$-distinguished if 
$Hom_H(\pi,\mathbb{C})\neq 0$. In this paper, we study the distinction of a 
special representation of $G$ called the Steinberg representation. We denote 
this representation by $St$. The Steinberg representation is defined both over 
a finite field and over a $p$-adic field.

 Our first result is a necessary condition for the distinction of $St$. The 
condition is that the symmetric space $X=G/H$ has to be quasi-split (see 
Definition \ref{qs}).

We also give a sufficient condition for distinction of $St$ under the further 
assumption of $H$ being split. We prove that in the case of a finite field, $X$ 
being quasi-split is also a sufficient condition.

In the case of a $p$-adic field, this is not a sufficient condition. In the 
example given by $G=GL_{2n+1}(F)$ and $H=GL_n(F)\times GL_{n+1}(F)$ for a $p$-adic
 field $F$, the space $X=G/H$ is quasi-split but $St$ is not $H$ 
distinguished. We prove that this example is essentially the unique such example
(see Theorem \ref{local field}). We give an interpretation 
of this result using the dual group of $X$ (see Theorem \ref{SV conjecture}).

\begin{subsection}{Main results}
Let $F$ be a field that is either a non archimedean local field of residual 
characteristic not 2, or a finite field of characteristic not 2.

Let $\mathbf{G}$ be a connected reductive group split over $F$ and let 
$G=\mathbf{G}(F)$. 
Let $\sigma:\mathbf{G}\rightarrow \mathbf{G}$ be an algebraic involution 
defined over $F$. We denote by $\mathbf{H}$ the fixed points of $\sigma$, and 
by $H=\mathbf{H}(F)$ the $F$-points of $\mathbf{H}$. We call
$X=G/H$ a symmetric space.

We discuss both the case of the Steinberg representation over a finite field 
and over a non archimedean local field.

In the case of $F$ being a finite field, there is a unique irreducible 
representation of $G$ called the Steinberg representation (introduced in 
\cite{St51}) denoted by $St$. In the case of $F$ being a local field, for any 
unramified character $\chi:G\rightarrow\mathbb{C}^\times$ there is a 
corresponding Steinberg representation  denoted by $St_{\chi}$ (see 
\cite{da054e0c-e380-3be7-b7aa-c246b458cfd4}).

An important notion that appears repeatedly in this work is the notion of 
$X$ being quasi-split.

\label{qs}
\begin{defn}
    We say that $X$ is quasi-split if $G$ has a Borel subgroup $B$ such that 
$\sigma(B)\cap B$ is a maximal torus of $B$.
\end{defn}

In the case of a finite field we prove the following.

\begin{theorem}\label{finte field} Let $F$ be a finite field. Under the 
assumption that both $G$ and $H$ are split over $F$, $St$ is $H$ distinguished 
if and only if $X$ is quasi-split. 
\end{theorem}

Now, we discuss the case of $F$ being a non archimedean local field. Let $\s$ 
be the ring of integers of $F$. We assume that both $G$ and $\sigma$ are 
defined over $\s$.

In this case, $X$ being quasi-split is not a sufficient condition for the 
distinction of the Steinberg representation. This can be seen from the 
following proposition.

\begin{prop}
    Let $F$ be a non archimedean local field. Let $G=GL_{2n+1}(F)$, 
$H=GL_n(F)\times GL_{n+1}(F)$, and $X=G/H$. We have the following:
    \begin{enumerate}
        \item $X$ is quasi-split.
        \item For any unramified character 
$\chi:G\rightarrow\mathbb{C}^\times$, the $G$ representation $St_\chi$ is not 
$H$-distinguished.
    \end{enumerate}
\end{prop}

From now on we assume that $H$ splits over $F$, we give a sufficient 
condition for the distinction of the Steinberg representation. For this we need 
more notations.
 
 Let $\mathbf{Z}=Z(\mathbf{G})$ be the center of $\mathbf{G}$.  We write 
$\mathbf{G}/\mathbf{Z}$ as a product of adjoint simple groups 
$\mathbf{G}/\mathbf{Z}=G_1\times...\times G_m$. 

Let $T$ be a maximal split torus of $G$, let $T^0$ be the maximal compact 
subgroup of $T$ and let $W_{aff}=N_G(T)/T^0$ be the affine Weyl group. 

We use the extended Bruhat Tits building of $G$, denote it by $\B$.

Let $l$ be the length function on $W_{aff}$. Denote by $\Omega=l^{-1}(0)$ the 
elements of length 0. Let $\omega:G\rightarrow \Omega$ be the
homomorphism given by acting on the $\Omega$ colors of chambers in $\B$, see Proposition 2.15 of 
\cite{my} for a precise definition.

We show that in the extended Bruhat Tits building of $G$ there is a chamber $\C_0$ which is 
$\sigma$ stable (i.e. $\sigma(\C_0)=\C_0)$.  Let $\Omega_{H,\C_0}=\{\omega(h)| h\in 
H,h\C_0=\C_0\}$ be a subgroup of $\Omega$. We prove that $\Omega_{H,\C_0}\subset 
\Omega$ does not depend on the choice of the $\sigma$ stable chamber $\C_0$ (see 
Proposition \ref{Omega H is well defined}). We denote $\Omega_H=\Omega_{H,\C_0}$ 
for some choice of $\C_0$.

The following character was introduced in the context of the distinction of the 
Steinberg representation by Broussous in  \cite{BP14}.
 \begin{defn}
     Let $\chi_0:G\rightarrow \mathbb{C}^\times$ be the quadratic unramified 
character defined by 

$\chi_0(g)=sgn(\omega(g))$, where $sgn$ is the sign of the 
permutation defined by $\omega(g)$ on the vertices of a chamber.
 \end{defn}

 Now we can state a sufficient condition for the distinction of the Steinberg 
representation.

 \begin{theorem}\label{local field}
     Let $F$ be a non archimedean local field, assume that both $G$ and $H$ are 
split over $F$. Let $\chi:G\rightarrow\mathbb{C}^\times$ be an unramified 
character, and let $St_\chi$ be the corresponding Steinberg representation.
      The representation $St_\chi$ is $H$-distinguished if all the following conditions hold:
     \begin{enumerate}
         \item $X$ is quasi-split.
         \item Write $\mathbf{G}/\mathbf{Z}=G_1\times...\times G_m$ a product 
of simple adjoint groups. There is no $\sigma$ invariant $G_i$, such that 
$G_i\cong PGL_{2n+1}$ and $G^\sigma_i\cong P(GL_n\times GL_{n+1})$.
         \item $\chi^{-1}\chi_0$ is trivial on $\Omega_H$. 
     \end{enumerate}
 \end{theorem}

\begin{Remark}
\begin{enumerate}
    \item Notice that for $\chi=\chi_0$ the last condition holds automatically.
    \item In the case of $\chi=\chi_0$. The opposite direction of Theorem 
\ref{local field} also holds. If $St_{\chi_0}$ is $H$-distinguished then 
conditions 1 and 2 hold.
  
\end{enumerate}
    
    \end{Remark}

We can interpret our results in terms of the dual group. Let $G^\vee$ be the 
dual group of $G$ and let $G^\vee_X$ be the dual group of $X$ which comes with 
a map $\iota:G^\vee_X\times SL_2\rightarrow G^\vee$ (see \cite{KS17} and 
\cite{takeda}).

 Let $\mathcal{U}$ be the unipotent variety of $G^\vee$ and let $\mathcal{U}_X$ 
be the unipotent variety of $G^\vee_X$.

 Conjecturally, the $H$ distinction of a representation is related to its Langlands 
parameter factoring through $\iota$. A version of this for $L^2$ distinction
is given in Conjecture 16.5.1 of \cite{SV17}, for the 
Steinberg representation smooth and $L^2$ distinction are equivalent.
 A conjecture for smooth distinction is given in Conjecture 1.1 of \cite{takeda}. 

We prove a version of these conjectures for the Steinberg representation.
In a different work \cite{shtotland2026relativekazhdanlusztigisomorphism} we study distinction of general representations
with an Iwahori fixed vector. For general representations the story is not as simple
as the story for the Steinberg representation 
(see Conjecture 1.14 of \cite{shtotland2026relativekazhdanlusztigisomorphism}).

\begin{theorem}\label{SV conjecture}
         Let $F$ be a non archimedean local field, assume that $G$ and $H$ are 
split over $F$.
     The representation $St_{\chi_0}$ is $H$ distinguished if and only if the following equivalent 
conditions hold:
     \begin{enumerate}
         \item $\iota(\mathcal{U}_X)$ intersects the open $G^\vee$ orbit of 
$\mathcal{U}$.
         \item The Langlands parameter $\phi_{St}$ factors through $G^\vee_X$.
     \end{enumerate}
\end{theorem}

\end{subsection}

\begin{subsection}{Methods of proof}

We introduce two hyper-graphs related to $X$. We relate the 
distinction problem of the Steinberg representation to problems about these 
hyper-graphs.

Let $B$ be a Borel subgroup of $G$, and let $\Delta$ be the set of simple roots 
corresponding to $B$. For any $\alpha\in\Delta$, let $P_\alpha$ be the standard 
parabolic subgroup corresponding to $\alpha$.

\begin{defn}\label{def graph borel}
    We define a hyper-graph $\Gamma_F(X)$ whose set of vertices, $B_F(X)$, is 
the set of $B$ orbits on $X$ over $F$. The hyper-edges of $\Gamma_F(X)$ are 
labeled by $\Delta$. For any $\alpha\in\Delta$ and $b\in B_F(X)$, we have the 
hyper-edge of Borel orbits $$E_{b,\alpha}=\{b'\in B_F(X)|P_\alpha b'=P_\alpha 
b\}.$$
\end{defn}

For the case of $F$ being a non archimedean local field, we define one more 
hyper-graph, denoted by $\Gamma_{aff,F}(X)$. Let $I$ be an Iwahori subgroup of 
$G$. Let $\Tilde{\Delta}$ be a set of simple reflections in the affine Weyl 
group $W_{aff}$ corresponding to $I$.

For any $\alpha\in \Tilde{\Delta}$, let $I_\alpha$ be the parahoric group 
containing $I$ and corresponding to $\alpha$.

\begin{defn}\label{def graph Iwahori}
    We define a hyper-graph $\Gamma_{aff,F}(X)$ whose set of vertices, 
$I_F(X)$, is the set of $I$ orbits on $X$. The hyper-edges of 
$\Gamma_{aff,F}(X)$ are labeled by $\Tilde{\Delta}$. For any 
$\alpha\in\Tilde{\Delta}$ and $x\in I_F(X)$ an $I$ orbit, we have the 
hyper-edge of $I$ orbits $E_{x,\alpha}=\{x'\in I_F(X)|I_\alpha x'=I_\alpha x\}$.
\end{defn}

The group $W_{aff}$ acts on the vertices of the hyper-graph $\Gamma_{aff,F}(X)$ 
(see Theorem 1.1 of \cite{my}). In particular $\Omega\subset W_{aff}$ acts on 
the vertices of $\Gamma_{aff,F}(X)$. The group $\Omega$ also acts on the 
hyper-edges of $\Gamma_{aff,F}(X)$. 

Denote by $\Gamma^0_{aff,F}(X)$ the quotient of $\Gamma_{aff,F}(X)$ by the 
action of $\Omega$. The vertices of $\Gamma^0_{aff,F}(X)$ can be identified 
with the $H$ orbits of chambers in $\B$.

\begin{defn}
    Let $\Gamma$ be a hyper-graph, a function $f:\Gamma\rightarrow\mathbb{C}$ 
is called harmonic if for every hyper-edge $E$, we have $\sum_{e\in E}f(e)=0$. 
Denote by $\mathcal{H}(\Gamma)$ the space of all harmonic functions on 
$\Gamma$.  
\end{defn}

The hyper-graphs we defined are related to the problem of determining the distinction of 
the Steinberg representation by the following proposition.

\begin{prop}
    \begin{itemize}
        \item For a finite field $F$ we have 
$dim(Hom_H(St,1))=dim\mathcal{H}(\Gamma_{F}(X))$.
        \item For a non archimedean local field $F$ we have 
$dim(Hom_H(St_{\chi_0},1))=dim\mathcal{H}(\Gamma^0_{aff,F}(X))$. 
    \end{itemize}
\end{prop}

Therefore, in order to show that the Steinberg representation is distinguished 
it is enough to construct a non-zero harmonic functions on these hyper-graphs. 
We provide a method to construct such functions.

First, we work over a field that is an algebraic closure of a finite field, 
we denote this field by $\Bar{F}$. We assume that $X$ is quasi-split and 
construct a non-zero harmonic function on $\Gamma_{\Bar{F}}(X)$. 

Next, we pull back the harmonic function on $\Gamma_{\Bar{F}}(X)$ to a harmonic 
function on $\Gamma_{F}(X)$. For this step we use the assumption that $H$ is 
split over $F$.

Lastly, we move to non archimedean local fields. Under conditions 1 and 2 of 
Theorem \ref{local field} we extend our construction of a harmonic function
 in the finite field case to a construction of a non-zero harmonic functions on $\Gamma^0_{aff,F}(X)$.

\end{subsection}

\begin{subsection}{Related works}
The problem of distinction of the Steinberg representation over a local field 
was studied extensively in the Galois case. By the Galois case we mean the case 
where $E/F$ is a quadratic extension, $G=\mathbf{G}(E)$, and $H=\mathbf{G}(F)$. 
The work on this case began with the paper of Prasad 
\cite{prasad2015relativelocallanglandscorrespondence}. For specific results, 
see, for example, \cite{MN17}, \cite{CM19}, \cite{BP14},\cite{CF17} and 
\cite{BP18}. All of the results mentioned in these works are in some sense 
orthogonal to ours as they deal with non split groups but only in the Galois 
case.

Interestingly, harmonic functions are also used in \cite{BP18} in order to 
analyze the distinction of the Steinberg representation, although on different 
graphs. 

We also mention the work \cite{wang2024distinctionsteinbergrepresentationrespect}
that also deals with the distinction of the Steinberg representation with 
respect to symmetric subgroups. In this work they 
prove an upper bound on the dimension of $Hom_H(St,\chi)$ and compute it for 
the case of $G$ being $GL_n$ and $H$ being the special orthogonal group. Both 
our methods and their methods are heavily based on the Bruhat Tits building. In 
fact, many of the combinatorial arguments related to the Bruhat Tits building 
used in \cite{my} also appear in 
\cite{wang2024distinctionsteinbergrepresentationrespect}. However, the main 
results of our work and the work 
\cite{wang2024distinctionsteinbergrepresentationrespect} are quite different. 
Most of our work is dedicated to proving the distinction of the Steinberg 
representation under some assumption (i.e. showing that $Hom_H(St,\chi)\neq 
0$), while most of the results of 
\cite{wang2024distinctionsteinbergrepresentationrespect} are about bounding the 
dimension of $Hom_H(St,\chi)$ from above.
The dual group of $X$ does not play a role in \cite{wang2024distinctionsteinbergrepresentationrespect}.
In some aspects, their results are more general than ours as they do not 
assume that the groups $G$ and $H$ are split.

\end{subsection}

\begin{subsection}{Structure of the paper}
In Section \ref{s2}, we recall the definition and relevant properties of the 
Steinberg representation.

In Section \ref{s3}, we relate the problems about distinction of the Steinberg 
representation to problems about existence of non-zero harmonic functions on 
the hyper-graphs $\Gamma_{F}(X)$ and $\Gamma_{aff,F}(X)$.

In Section \ref{s4}, we prove that $X$ being quasi-split is a necessary 
condition for distinction of the Steinberg representation. 

In Section \ref{s5}, we determine when a non-zero harmonic function exists on 
the hyper-graph $\Gamma_{\Bar{F}}(X)$. Here, $\Bar{F}$ is an algebraic closure 
of a finite field $F$.

In Section \ref{s6}, we discuss harmonic function on $\Gamma_F(X)$ for a finite 
field $F$ and we prove Theorem \ref{finte field}.

In Section \ref{s7}, we discuss harmonic function on $\Gamma_{aff,F}(X)$ and we 
prove Theorem \ref{local field} for the case $\chi=\chi_0$.

In Section \ref{s8}, we extend the results of Section \ref{s7} to a general 
unramified character and finish the proof of Theorem \ref{local field}.

In Section \ref{s9}, we relate our results to the dual group $G^\vee_X$ and 
prove Theorem \ref{SV conjecture}.

\end{subsection}

\begin{subsection}{Acknowledgement}
I would like to thank Raphaël Beuzart-Plessis for his suggestion of applying the 
methods of \cite{my} to the problem of the distinction of the Steinberg 
representation. I would like to thank my advisor Eitan Sayag for helpful 
discussions, and Yiannis Sakellaridis for his comments on an earlier version
of this manuscript. I was partially supported by ISF grant number 1781/23 during the work 
on this paper.

\end{subsection}

\end{section}

\begin{section}{The Steinberg representation}\label{s2}

The Steinberg representation is defined both for a group over a finite field 
and for a group over a non archimedean local field.

\begin{subsection}{The Steinberg representation over a finite field}

\begin{defn}[Subsection 1.5 of \cite{Humphreys}]
    Let $F$ be a finite field and let $\B_{sph}$ be the spherical building of 
$G$. Let $r_s$ be the semi-simple rank of $G$. The Steinberg representation of 
$G$ (denoted by $St$) can be defined using the $r-1$ homology group of 
$\B_{sph}$ with the action of $G$ on it, i.e. $St=H_{r-1}(\B_{sph},\Q)$.
\end{defn}

\begin{prop}[See \cite{Humphreys}]
    $St$ is irreducible.
\end{prop}

\begin{prop}[Subsection 1.4 of \cite{Humphreys}]
    Let $F$ be a finite field, let $B$ be a Borel subgroup of $G$ and let 
$H(G,B)$ be the Hecke algebra of $G$. 
This is the algebra of $B$ bi-invariant functions on $G$.
 The space of fixed points, $St^B$, 
is one dimensional and $H(G,B)$ acts on $St^B$ by $sgn$. That is, for every 
simple reflection $s\in\Delta$, the element $T_s=1_{BsB}$ acts on $St^B$ by 
$-1$. 
\end{prop}

\end{subsection}
\begin{subsection}{The Steinberg representation over a non archimedean local 
field}
    
\end{subsection}

\begin{defn}
    Let $F$ be a non archimedean local field and let $\B_{aff}$ be the affine 
building of $G$. 
    The Steinberg representation, $St$, of $G$ is defined to be the top compact 
cohomology group of $\B_{aff}$ with the action of $G$ on it, 
$St=H^{top}_c(\B_{aff})$
\end{defn}

\begin{defn}
    Let $\chi:G\rightarrow\mathbb{C}^\times$ be an unramified character. We 
define $St_\chi=St\otimes \chi$.
\end{defn}

\begin{prop}[See \cite{BP18}]
    The representation $St$ is irreducible.
\end{prop}

Let $W_{aff}$ be the affine Weyl group of $G$ and let $\Omega$ be the subgroup 
of $W_{aff}$ of elements of length $0$. The group $\Omega$ can be embedded in the 
permutation group of the vertices of any chamber $\C\in \B$.  There is a 
homomorphism $\omega:G\rightarrow \Omega$ which can be defined by comparing the 
order of the vertices of the alcoves $g\C$ and $\C$. See Proposition 2.15 of 
\cite{my} for a precise definition.

\begin{prop}\label{Steinberg gives sgn}[See Subsection 3.3 of \cite{BP18}]
    Let $I$ be an Iwahori subgroup of $G$ and let $H(G,I)$ be the affine 
Hecke algebra of $G$. This is the algebra of $I$ bi-invariant compactly supported
 functions on $G$. The module $St^I$ is one dimensional and $H(G,I)$ acts on $St^I$ by 
$sgn$. That is, for every simple reflection $s\in\Tilde{\Delta}$, the element 
$T_s=1_{IsI}$ acts on $St^I$ by $-1$. For $o\in \Omega$, the element 
$T_o=1_{IoI}$ acts on $St^I$ by $sgn(o)$, the sign of $o$ as a permutation on 
the vertices of $\C$.
\end{prop}

It follows from the last proposition that $St$ is in the principal block of $G$ 
representations generated by their $I$ fixed vectors (see \cite{Borel1976}).

\end{section}

\begin{section}{Distinction of the Steinberg representations and harmonic 
functions}\label{s3}
    In this Section we relate the questions about the distinction of the 
Steinberg representation to questions about the existence of a non-zero 
harmonic function on the hyper-graphs $\Gamma_F(X)$ and $\Gamma_{aff,F}(X)$ 
defined in Definitions \ref{def graph borel} and \ref{def graph Iwahori} 
respectively.

    We begin with the case of a finite field.
    
\begin{prop}\label{harmonic finite}
        \item Let $F$ be a finite field, we have 
$dim(Hom_H(St,1))=dim\mathcal{H}(\Gamma_{F}(X))$.        
\end{prop}

\begin{proof}
    By Frobenius reciprocity we have 
$Hom_H(St,1)=Hom_G(St,ind^G_H(1))=Hom_G(St,\mathbb{C}[X])$. Here, 
$\mathbb{C}[X]$ is the $G$ representation of functions on $X$. 

    Let $B$ be a Borel subgroup of $G$. The representation $St$ is irreducible and has a $B$ fixed 
vector. We can write $\mathbb{C}[X]$ as a direct sum of its sub-representation 
generated by its $B$ fixed vectors and some complement, 
$\mathbb{C}[X]=G\mathbb{C}[X]^B\oplus \pi$. Any map from $St$ to 
$\mathbb{C}[X]$ lands in $G\mathbb{C}[X]^B$, so 
$Hom_G(St,\mathbb{C}[X])=Hom_G(St,G\mathbb{C}[X]^B)$.

    Both the Steinberg representation and $G\mathbb{C}[X]^B$ are self dual 
representations of $G$, therefore 
$Hom_G(St,G\mathbb{C}[X]^B)=Hom_G(G\mathbb{C}[X]^B,St)$.

    We consider $\mathbb{C}[X]^B$ and $St^B$ as modules of the Hecke algebra 
$H(G,B)$ of $B$ bi-invariant functions on $G$. We have 
$Hom_G(G\mathbb{C}[X]^B,St)=Hom_{H(G,B)}(\mathbb{C}[X]^B,St^B)$.

    We claim that $Hom_{ 
H(G,B)}(\mathbb{C}[X]^B,St^B)=\mathcal{H}(\Gamma_F(X))$. 
    
    The module $St^B$ is one dimensional, in fact it is the sign representation 
of $H(G,B)$. Any element of $Hom_{H(G,B)}(\mathbb{C}[X]^B,St^B)$ is a 
function on the $B$ orbits on $X$.

    Let $f\in Hom_{H(G,B)}(\mathbb{C}[X]^B,St^B)$, for every 
$s_\alpha\in \Delta$, a simple reflection, the element 

$T_\alpha=1_{Bs_\alpha B}\in H(G,B)$ acts as $-1$ on $f$, i.e. $T_\alpha f=-f$.
 Thus, for every Borel orbit 
$b\in B\backslash X$ we have $((T_s+1)f)(b)=0$. This is exactly the condition 
$\sum_{b'\in E_{b,\alpha}}f(b')=0$. Therefore $f$ is a harmonic function on 
$\Gamma_F(X)$. The same argument also works in the opposite direction.
\end{proof}

Now we move to the local field case. The proof is almost the same.

\begin{prop}\label{steinberg harmonic zero}
Let $F$ be a non archimedean local field, we have  

$dim(Hom_H(St_{\chi_0},1))=dim\mathcal{H}(\Gamma^0_{aff,F}(X))$.    
\end{prop}

\begin{proof}
    By Frobenius reciprocity we have 
$$Hom_H(St_{\chi_0},1)=Hom_G(St_{\chi_0},Ind^G_H(1))=Hom_G(St_{\chi_0},C^\infty(
X))$$
    
    Here, $C^\infty(X)$ is the $G$ representation of locally constant functions on 
$X$. 

    The representation $St_{\chi_0}$ is self dual, therefore 
$Hom_G(St_{\chi_0},C^\infty(X))=Hom_G(S(X),St_{\chi_0})$. Here, $S(X)$ is the 
$G$ representation of locally constant compactly supported functions on $X$.

    Let $I$ be an Iwahori subgroup of $G$. $St_{\chi_0}$ is irreducible and is 
generated by its $I$ fixed vectors. The subgroup $I$ splits the category of 
smooth representations of $G$ (see \cite{Bcenter}), and the category of 
representations generated by their $I$ fixed vectors is equivalent to the 
category of $\mathcal{H}(G,I)$ modules (see \cite{Borel1976},\cite{Matsumoto1977AnalyseHD}). 

Therefore $Hom_G(S(X),St_{\chi_0})=Hom_{H(G,I)}(S(X)^I,St_{\chi_0}^I)$.

    We claim that $Hom_{H(G,I)}(S(X)^I,St_{\chi_0}^I)=\mathcal{H}(\Gamma^0_{aff,F}(X))$. 
    The $H(G,I)$ module $St_{\chi_0}^I$ is one dimensional, in fact it is the sign 
representation of $H(G,I)$. 
Any element of $Hom_{H(G,I)}(S(X)^I,St_{\chi_0}^I)$ is a function on the $I$ orbits on $X$.
    
    Let $f\in Hom_{H(G,I)}(S(X)^I,St_{\chi_0}^I)$. For any simple 
reflection $s\in \Tilde{\Delta}$, the element 

$T_s=1_{IsI}\in H(G,I)$ acts as 
$-1$ on $f$, $T_sf=-f$. Thus, for every $I$ orbit $b\in I\backslash X$ we have 
$((T_s+1)f)(b)=0$, this is exactly the condition $\sum_{b'\in E_{b,s}}f(b')=0$ 
(see Section 4 of \cite{my}). Now we just need to show that $f$ is constant on 
$\Omega$ orbits to get that it is a harmonic function on $\Gamma^0_{aff,F}(X)$.

    We use the action of the elements $T_o\in H(G,I)$ for $o\in \Omega$. 
    
    By Proposition \ref{Steinberg gives sgn} we have $(T_of)(b)=\chi^2_0(o)f(b)=f(b)$.
     We also have $(T_of)(b)=f(o^{-1}\times b)$ (see Proposition 4.1 in \cite{my}). 
    Therefore, $f$ is constant on $\Omega$ orbits.

    The same argument also works in the opposite direction.
    
\end{proof}

In Proposition \ref{steinberg harmonic zero} we related the question about the distinction of 
$St_{\chi_0}$ to a question about the existence of non-zero harmonic functions. 
Next, we want to relate the question about the distinction of $St_\chi$ for a 
general unramified character $\chi:G\rightarrow\mathbb{C}^\times$ to a question 
about the existence of non-zero harmonic functions.

The vertices of $\Gamma_{aff,F}(X)$ can be identified with the set of $H$ 
orbits on $\Omega$ colored chambers in the extended Bruhat Tits building of 
$G$, which we denote by $\B$ (see Section 3 of \cite{my}). The vertices of 
$\Gamma^{0}_{aff,F}(X)$ can be identified with the set of $H$ orbits on the 
chambers in $\B$. 

We introduce one more hyper-graph. Let $\Gamma_{aff,F}^1(X)$ be the sub 
hyper-graph of $\Gamma_{aff,F}(X)$ whose vertices correspond to $H$ orbits of 
chambers colored by the unit $1\in\Omega$. The vertices of  
$\Gamma_{aff,F}^1(X)$ can be identified with $H\cap Ker(\omega)$ orbits on the 
chambers of $\B$. Denote $H^1=H\cap Ker(\omega)$.

\begin{prop}\label{distinction harmonic general}
    $St_\chi$ is $H$ distinguished if and only if there is a non-zero 
$\phi\in\mathcal{H}(\Gamma^1_{aff,F}(X))$ such that for any chamber $\C$ and 
$h\in H$, $\phi(H^1h\C)=\chi^{-1}\chi_0(\omega(h^{-1}))\phi(H^1\C)$. 
\end{prop}

\begin{proof}
    Like in the proof of Proposition \ref{steinberg harmonic zero}, we have 
$Hom_H(St_\chi,1)=Hom_{H(G,I)}(S(X)^I,St_{\chi^{-1}}^I)$ which can be described as a subspace 
of the space of harmonic functions $\phi$ on $\Gamma_{aff,F}(X)$. The extra 
condition on $\phi$ to belong to this subspace is the following. Let 
$o\in\Omega$ and $\C\subset \B$ be a chamber. Denote by $x_{\C,o}$ the $I$ 
orbit on $X$ corresponding to the chamber $\C$ colored by the color $o$. The 
condition is $\phi(x_{\C,o})=\chi_0(o)\chi^{-1}(o)\phi(x_{\C,1})$. We 
refer to this condition as condition $(*)$.

    Now we show that this space of functions is the same as the space of functions

    $\phi\in\mathcal{H}(\Gamma^1_{aff,F}(X))$ such that for any chamber $\C$ and 
$h\in H$, $\phi(H^1h\C)=\chi^{-1}\chi_0(\omega(h^{-1}))\phi(H^1\C)$. 
    
    For the first direction, assume that there exists  a non-zero harmonic 
function $\phi\in \mathcal{H}(\Gamma_{aff.F})$ satisfying condition $(*)$. 
Restricting $\phi$ to $\Gamma^1_{aff,F}$ gives us a harmonic function on 
$\Gamma^1_{aff,F}(X)$.

    Clearly, if $h\in H$ and $\C$ a chamber then 
$$\phi(H^1h\C)=\phi(x_{h\C,1})=\phi(x_{\C,\omega(h^{-1})})=\chi_0\chi^{-1}(\omega(h^{-1}))\phi(x_{\C,1})=\chi_0\chi^{-1}(\omega(h^{-1}))\phi(H^1\C)$$

    For the second direction, let $\phi\in \Gamma^1_{aff,F}(X)$ be such that for 
    any chamber $\C$ and $h\in H$, $\phi(H^1h\C)=\chi^{-1}\chi_0(\omega(h^{-1}))\phi(H^1\C)$.
     We define a function $\phi'$ on $\Gamma_{aff,F}(X)$ by assigning to a $H$ orbit of a 
chamber $\C$ with the color $o$ the value 
$\phi'(\C,o)=\phi(H^1\C)\chi^{-1}\chi_0(o)$. We only need to check that this 
function is $H$ invariant. Let $h\in H$, $o\in\Omega$ and $\C$ a chamber,  
$\phi'(h\C,\omega(h)o)=\chi^{-1}\chi_0(\omega(h)o)\phi(H^1h\C)=\chi^{
-1}\chi_0(o)\phi(H^1\C)=\phi'(\C,o)$. 
\end{proof}

\end{section}

\begin{section}{$X$ being quasi-split is a necessary condition for distinction 
}\label{s4}

In this Section we prove that $X$ being quasi-split is a necessary 
condition for distinction in the case of a non archimedean local field. With slight 
modifications, our proof can be applied also to the case of a finite field.

\begin{Remark}
    The main result of this Section follows from known results. It is known that $St$ is 
generic, i.e. has a Whittaker model (see for example \cite{BBBG24}). By 
\cite{prasad2019generic} a generic representation can be distinguished only if 
$X$ is quasi-split. 

    We provide a direct proof.
\end{Remark}

Let $F$ be a non archimedean local field and let $\B$ be the extended Bruhat Tits building of 
$G$.

The involution $\sigma$ acts on $\B$ and that an apartment, $\A$, of $\B$ is 
called $\sigma$ stable if $\sigma(\A)=\A$ (see Subsection 2.5 of \cite{my}).

We have the notion of the $l_\sigma$ length of a chamber (see Definition 5.1 of \cite{my}). 
It is defined as $l_\sigma(\C)=d(\C,\sigma(\C))$, 
where $d$ is the distance between chambers in the building.

We introduce a partial order on the $\sigma$ stable apartments. 

\begin{defn}
    Let $\A$ be a $\sigma$ stable apartment, and let $T\subset G$ be a maximal 
split torus corresponding to $\A$. Define $rank(\A)=rank(T)-rank(T\cap H)$.
\end{defn}

\begin{defn}\label{def order}
    We define a partial order $<$ to be the minimal partial order such that for 
any two $\sigma$ stable apartments $\A_1,\A_2$ with $rank(\A_1)<rank(\A_2)$ and 
$\A_1\cap \A_2$ of codimension 1 inside $\A_1$ (equivalently $\A_2$) we have 
$\A_1<\A_2$. 

    This partial order descends to a partial order on the $H$ orbits of 
$\sigma$ stable apartments.
 \end{defn}

\begin{prop}\label{apartment order}
    Let $\A_1,\A_2$ be two $\sigma$ stable apartments such that $H\A_1\neq 
H\A_2$ and such that their intersection is of codimension 1 inside them. Then 
$|rank(\A_1)-rank(\A_2)|\leq 1$. Furthermore, $\A_1<\A_2$ if and only if the 
action of $\sigma$ on $\A_1$ preserves the two half planes defined by $\A_1\cap 
\A_2\subset \A_1$.
\end{prop}

\begin{proof}
    Let $\A$ be a $\sigma$ stable apartment, $rank(\A)$ is equal to the 
co-dimension of the subgroup of translations on the affine space $\A$ that 
commute with the action of $\sigma$ on $\A$ (see Proposition 5.7 in \cite{my}). 

    Now, let $\A_1,\A_2$ be two $\sigma$ stable apartments whose intersection 
is of codimension 1. Denote by $d$ the dimension of the group of translation on 
$\A_1\cap \A_2$ that commute with $\sigma$. 
    
    Denote $n=dim\A_1=dim\A_2$. Then, for $i\in \{1,2\}$, $rank(\A_i)$ is 
either $n-d$ or $n-d-1$. If the translations on $\A_i$ that are in a direction perpendicular to 
$\A_1\cap \A_2$ commute with $\sigma$ then $rank(\A_i)=n-d-1$, otherwise 
$rank(\A_i)=n-d$.

    Let $\C$ be a chamber in $\A_i$ that contains a facet of codimension 1 that 
spans $\A_1\cap \A_2$ (as an affine hyper plane). The translations perpendicular 
to $\A_1\cap \A_2$ commute with $\sigma$ if and only if $\sigma(\C)$ and $\C$ 
are on the same side of $\A_1\cap \A_2$.

By Remark 7.10 of \cite{my}, this 
occurs if and only if $\C$ has minimal $l_\sigma$ length among the chambers 
that contain the facet $f=\C\cap\A_1\cap\A_2$. 
    
    There is at most one $H$ orbit of $\sigma$ stable apartments 
that contains a chamber $\C$ that satisfies the above property (see Section 7 of 
\cite{my}). This finishes the proof as $\sigma(\C)$ and $\C$ 
are on the same side of $\A_1\cap \A_2$ if and only if $\sigma$ 
preserves the two half planes defined by $\A_1\cap \A_2$.
    
\end{proof}

 \begin{defn}
     We say that $X$ is quasi-split if there is a Borel subgroup, $B$, of $G$ 
such that $B\cap\sigma(B)$ is a maximal torus.
 \end{defn}

 We use an equivalent condition that we now state.

 \begin{defn}
     A torus $T\subset G$ is called $\sigma$ split if 
 it is $\sigma$ stable and $\sigma$ acts on it by inverse. 
 
 For any torus $T$, we define its $\sigma$ split part, $T^-=\{t\in T|\sigma(t)=t^{-1}\}^0$, 
 the connected component of the identity.
 \end{defn}

 \begin{prop}[\cite{V74}]
     $X$ is quasi-split if and only if there is a maximal $\sigma$ split torus 
$T$ such that $Z_G(T)$ is a maximal torus of $G$.  
 \end{prop}

By Proposition \ref{distinction harmonic general} in order to prove that $X$ being 
quasi-split is a necessary condition for the distinction of $St_{\chi}$, it 
is enough to prove the next result.

\begin{theorem}\label{easy direction}
    If $X$ is not quasi-split then $\mathcal{H}(\Gamma^1_{aff,F}(X))=0$.
\end{theorem}

Let $T$ be a maximal torus of $G$ and let $\alpha:T\rightarrow F^\times$ be a 
root of $T$. We denote by $u_\alpha:\mathbb{G}_a\rightarrow \mathbf{G}$ the 
isomorphism onto the unipotent subgroup $\mathbf{U}_\alpha\subset \mathbf{G}$ 
corresponding to $\alpha$ (see \cite{LAG} Proposition 21.9).

In order to prove the above Theorem we need the following proposition.

\begin{prop}\label{non quasi splits gives roots}
    Assume that $X$ is not quasi-split. For every maximal split $\sigma$ stable torus 
$T$ there is a root $\alpha:T\rightarrow F^\times$ such that for every $x\in 
F$, both $u_{\alpha}(x)$ and $u_{-\alpha}(x)$ are in $H$. 

We call such a root a good root.
\end{prop}

\begin{proof}
    We prove this by induction on the partial order defined in Definition 
\ref{def order}, regarded as a partial order on the set of maximal split $\sigma$ stable tori.
 We begin with the maximal tori, that are $\sigma$ stable 
split tori $T$ such that $T^-$ is a maximal $\sigma$ split torus.
    
    By our assumption, $X$ is not quasi-split. Therefore, $M=Z_G(T^-)$ is not a 
torus. $M$ is a Levi subgroup of $G$ (see for example 
Proposition 20.4 in \cite{LAG}). Let $\alpha^\vee:F^\times\rightarrow T$ be a 
coroot whose unipotent subgroup $U_\alpha$ is contained in $M$. We claim that 
$\sigma(\alpha^\vee)=\alpha^\vee$. For $t\in T^-$ and $x\in F$, we have 
$u_\alpha(x)=t^{-1}u_\alpha (x)t= u_\alpha(\alpha(t)x)$, so $\alpha(t)=1$ for 
$t\in T^-$. Thus, $\sigma(\alpha)$ and $\alpha$ are equal on $T^-$ and they are 
clearly equal on the $\sigma$ fixed part. Therefore, we have $\alpha=\sigma(\alpha)$ and
 $\sigma(\alpha^\vee)=\alpha^\vee$.

    We claim that for every $x\in F$, we have $u_\alpha(x)\in H$. We know that 
either $\sigma(u_\alpha(x))=u_\alpha(x)$ or $\sigma(u_\alpha(x))=u_\alpha(-x)$, 
as we showed that $\alpha$ is fixed by $\sigma$. We assume by contradiction 
that for every $x\in F$, $\sigma(u_\alpha(x))=u_\alpha(-x)$. It follows that 
also for every $x\in F$, $\sigma(u_{-\alpha}(x))=u_{-\alpha}(-x)$.

    Let $y=u_\alpha(\frac{1}{2})u_{-\alpha}(-1)$. We show that the torus 
$y^{-1}Ty$ contains a $\sigma$ split torus larger than $T^-$ and this is a 
contradiction. We have $y\in M$ so clearly $y^{-1}Ty$ contains $T^-$. It is 
enough to show that for every $x\in F^\times$, $\sigma$ acts as inverse on  
$y^{-1}\alpha^\vee(x)y$. 

    We need to show that 

    $$\sigma(y^{-1}\alpha^\vee(x^{-1})y)=y^{-1}\alpha^\vee(x)y$$

     We have $\sigma(y)=u_\alpha(-\frac{1}{2})u_{-\alpha}(1)$, so we need to show that

$$u_{-\alpha}(-1)u_\alpha(\frac{1}{2})\alpha^\vee(x^{-1})u_\alpha(-\frac{1}{2})u
_{-\alpha}(1)=u_{-\alpha}(1)u_\alpha(-\frac{1}{2})\alpha^\vee(x)u_\alpha(\frac{1
}{2})u_{-\alpha}(-1)$$

    This is equivalent to showing

$$u_\alpha(\frac{1}{2})u_{-\alpha}(-2)u_\alpha(\frac{1}{2})\alpha^\vee(x^{-1})u_
\alpha(-\frac{1}{2})u_{-\alpha}(2)u_\alpha(-\frac{1}{2})=\alpha^\vee(x)$$

    Using 
$u_\alpha(\frac{1}{2})u_{-\alpha}(-2)u_\alpha(\frac{1}{2})=n_\alpha\alpha^\vee(\frac{1}{2})$
 where $n_\alpha=u_\alpha(1)u_{-\alpha}(-1)u_\alpha(1)$ (see Lemma 8.1.4 
in \cite{Springer1981LinearAG}) we are left with the known result (again, see 
Lemma 8.1.4 in \cite{Springer1981LinearAG}): 

    $$n_\alpha\alpha^\vee(x^{-1})n_\alpha^{-1}=\alpha^\vee(x)$$.

    Thus, we proved the proposition for tori that are maximal with respect to 
our partial order.

     Let $\A_1,\A_2$ be two apartments such that $\A_1<\A_2$ and $\A_1\cap\A_2$ is of 
codimension 1 inside $\A_2$. Let $T_1$, $T_2$ be the maximal split tori 
corresponding to $\A_1$,$\A_2$ respectively. We may assume that we know the 
proposition for $T_2$, and prove it for $T_1$.

    Let $\alpha$ be a root of $T_2$ such that for every $x\in F$, 
$u_\alpha(x),u_{-\alpha}(x)\in H$. Let $\beta$ be a root of $T_2$ that is 
constant on $\A_1\cap \A_2$. There is an element $g$ in the group generated by 
$\{u_\beta(y),u_{-\beta}(y)|y\in F\}$ such that $g\A_2=\A_1$. It is enough to 
show that $g$ commutes with $u_{\pm\alpha}(x)$, as the conjugation of $\alpha$ 
by $g$ is a root of $T_1$, and if $g$ commutes with $u_{\pm\alpha}(x)$ then 
$g^{-1}u_{\pm\alpha}(x)g=u_{\pm\alpha}(x)\in H$. It is enough to show that 
$\alpha,\beta$ are perpendicular. 
    
    By Proposition \ref{apartment order}, $\sigma$ acts by $-1$ on $\beta$. We know that $\sigma$ 
acts as 1 on $\alpha$, therefore

$\langle\alpha,\beta\rangle=\langle\sigma(\alpha),\sigma(\beta)\rangle=-\langle\alpha,\beta\rangle$ and we are done.
    
\end{proof}

The following notion will be useful during the proof of Theorem \ref{easy 
direction}.

\begin{defn}
    
Let $\A$ be a $\sigma$ stable apartment corresponding to a maximal split 
torus $T$. Let $\alpha$ be a good root on $T$ like in Proposition \ref{non quasi splits gives roots}. 
We call a chamber $\C$ in $\A$ $\alpha$ adjacent if it contains a facet of 
codimension 1 that is part of a wall defined by an affine translation of 
$\alpha$. We call a chamber $H$-wall adjacent if it is $\alpha$ adjacent for 
some good root $\alpha$. 

Notice that if $h\in H$ fixes $\A$ then it acts on the set of $H$-wall 
adjacent chambers in $\A$. Also, $\sigma$ acts on the set of $H$-wall adjacent chambers.
\end{defn}

Now we can prove Theorem \ref{easy direction}.

\begin{proof}[Proof of Theorem \ref{easy direction}]

Let $\phi\in \mathcal{H}(\Gamma^1_{aff,F}(X))$, we show that it is zero. 
Recall that the vertices of $\Gamma^1_{aff,F}(X)$ can be identified with 
$H^1=H\cap Ker(\omega)$ orbits on the chambers of $\B$. We think about $\phi$ 
as a function on the chambers in $\B$, fixed on $H^1$ orbits. We prove 
that $\phi$ is zero on all chambers.

Let $\A$ be a $\sigma$ stable apartment corresponding to a maximal split 
torus $T$. Let $\alpha$ be a good root on $T$ given by the Proposition \ref{non quasi splits gives roots}.

The function $\phi$ is zero on all $\alpha$ adjacent chambers as edges labeled by 
$\alpha$ that contain an $\alpha$ adjacent chamber have size 1.

 All the walls defined by $\alpha$ are perpendicular to the fixed points of 
$\sigma$ on $\A$. 
 
 For every chamber $\C$ in $\A$ we choose a gallery $(\C_i)^n_{i=1}$ in $\A$, 
with $\C_1$ being $H$-wall adjacent and $\C_n=\C$. We may assume that $n$ is 
the minimal length of such a gallery.

 We prove that $\phi(\C_i)=0$ for every $i$. We have already discussed 
this for $i=1$. It is enough to show that for every $i$ we have 
$\phi(\C_i)+\phi(\C_{i+1})=0$. 

Notice that for every $i$, $H^1\C_i\neq H^1\C_{i+1}$ otherwise we could have 
found a shorter gallery.

If the hyper-edge that contains the $H^1$ orbits of $\C_i$ and $\C_{i+1}$ is of 
size 2, then $\phi$ being harmonic implies $\phi(\C_i)+\phi(\C_{i+1})=0$. Assume it is of a size larger than 
2. In this case, there is a $\sigma$ stable apartment $\A'$ that contains 
$\C_i\cap\C_{i+1}$ and $H\A\neq H\A'$. We claim that $\A<\A'$. By Proposition \ref{apartment 
order} it is enough to show that $\sigma$ preserves the two half planes defined 
by $\C_i\cap \C_{i+1}$ in $\A$. Assume otherwise, we claim that the reflection 
of $\C_1$ across $\C_i\cap \C_{i+1}$ is also a $H$-wall adjacent chamber. This 
holds because $\sigma$ acts as $-1$ on the root defining the wall spanned by $\C_i\cap\C_{i+1}$
and so all the good roots are perpendicular to $\C_i\cap\C_{i+1}$.

Thus, we can find another $H$-wall adjacent chamber, the reflection of $\C_1$ 
across $\C_i\cap \C_{i+1}$, whose distance to $\C_{i+1}$ is less than the 
distance between $\C_1$ and $\C_{i+1}$. Using this chamber we can construct a 
smaller gallery, a contradiction.

We proved that $\A<\A'$. Now we use induction on the partial order defined in Definition \ref{def order}. 
We know by induction that $\phi$ is zero on all 
chambers in $\A'$. The sum of the values of $\phi$ over all chambers in the 
hyper-edge containing $\C_i$ and $\C_{i+1}$ is 0. 
Every chamber in this hyper-edge except $\C_i$ and $\C_{i+1}$ is in an apartment 
such as $\A'$, meaning, an apartment containing $\C_i\cap \C_{i+1}$ and satisfying $H\A\neq H\A'$. 
Therefore, $\phi$ is zero on all chambers in this hyper-edge except $\C_i,\C_{i+1}$,
thus $\phi(\C_i)+\phi(\C_{i+1})=0$.
\end{proof}

\end{section}

\begin{section}{Harmonic functions in the algebraically closed field 
case}\label{s5}

In the previous Section we proved a necessary condition for the existence of a 
non-zero harmonic function on $\Gamma^1_{aff,F}(X)$. In this Section we begin the 
work for the second direction, which entails the construction of harmonic 
functions on our hyper-graphs. Let $\Bar{F}$ be the algebraic closure of 
 a finite field $F$. In this section we construct a non-zero harmonic 
function on $\Gamma_{\Bar{F}}(X)$ under the assumption that $X$ is quasi-split. 
In the following sections we upgrade this construction to 
$\Gamma_{F}(X)$ and to $\Gamma_{aff,F}(X)$.

\begin{defn}
    Let $\Gamma$ be a hyper-graph whose hyper-edges are labeled by a finite set 
$L$, let $L'\subset L$. Let $v$ be a vertex of $\Gamma$. We call $v$ $L'$ full 
if for every label $\alpha\in L'$ the hyper-edge with label $\alpha$ that 
contains $v$ is of size at least two. In the case $L'=L$ we refer to a $L$ full 
vertex as a full. 
\end{defn}

The following easily follows from a result of Prasad \cite{prasad2019generic}.

\begin{prop}\label{springer} (Proposition 9 of \cite{prasad2019generic})
    $X$ is quasi-split if and only if there is a closed Borel orbit that is 
full in $\Gamma_{\Bar{F}}(X)$. 
\end{prop}

\begin{Remark}
    A natural question to ask is whether the previous result holds also for 
spherical varieties. This will be a first step in extending our results 
from symmetric spaces to general spherical varieties. In the general case, the notion of being 
quasi-split is replaced with the condition that the stabilizer of the open Borel 
orbit of $X$ is a Borel subgroup of $G$. Such spherical varieties are called tempered.
\end{Remark}

We prove the following result.
\begin{prop}\label{dimesnion}
    The number of full closed Borel orbits in $\Gamma_{\Bar{F}}(X)$ is equal to 
$dim\mathcal{H}(\Gamma_{\Bar{F}}(X))$.
\end{prop}

From Propositions \ref{springer} and \ref{dimesnion} it follows that $X$ is 
quasi-split if and only if there is a non-zero harmonic function on 
$\Gamma_{\Bar{F}}(X)$.

Denote by $S$ the set of full closed Borel orbits in $\Gamma_{\Bar{F}}(X)$. The 
main ingredient in the proof of Proposition \ref{dimesnion} is the following.

\begin{prop}\label{support function}
    There exists a unique function on Borel orbits $s:B\backslash X\rightarrow \Z[S]$ 
such that:
    \begin{enumerate}
        \item For any $x\in S$ we have $s(x)=x$.
        \item For any $E$, a hyper-edge of $\Gamma_{\Bar{F}}(X)$, $\sum_{v\in 
E}s(v)=0$.
    \end{enumerate}
\end{prop}

Before proving Proposition \ref{support function} we use it to deduce Proposition
 \ref{dimesnion}.

 \begin{proof}[Proof of Proposition \ref{dimesnion}]
By composing the map $s$ from the Proposition \ref{support function} with any function 
$f:\Z[S]\rightarrow \mathbb{C}$ we obtain a harmonic function on 
$\Gamma_{\Bar{F}}(X)$. Thus, the dimension of harmonic functions is at least the 
number of closed full orbits. 

Assume that the dimension of harmonic functions is larger 
than the number of full closed orbits. In that case, there is a harmonic function that 
is zero on all full closed orbits. Such a function is zero on all the closed 
orbits. We claim that it must be zero on the entire hyper-graph. 
This follows by induction on the Bruhat order, i.e. 
the partial order given by closure relations of the Borel orbits.
 The Bruhat order restricted to a hyper-edge can be described using 
 the Bruhat order on a homogeneous spherical variety of $PGL_2$ (see 
\cite{knop}). Every homogeneous spherical variety of $PGL_2$ has a unique non closed orbit
 and therefore a harmonic function that is zero on all closed orbits must be zero.
    
\end{proof} 

\begin{proof}[Proof of Proposition \ref{support function}]
    Uniqueness is clear by induction on Bruhat order. 
    
    We construct $s$ by induction on the Bruhat order on $B\backslash X$. We 
define it on the full closed orbits by the first condition and we define it to 
be zero on the closed orbits that are not full.

    Let $v$ be a Borel orbit and assume that $s$ was already defined on all 
orbits that are smaller than $v$ in the Bruhat order. Let $\alpha,\beta\in\Delta$ 
be two roots such that the hyper-edges containing $v$ defined by them contain 
only orbits that are equal or less than $v$ (in the Bruhat order). There are 
four options for each of them, they can be roots of type $G,T,U$ or $N$ (see 
\cite{knop}). 

    Let $\Gamma_{v,\{\alpha,\beta\}}$ be the sub hyper-graph containing 
hyper-edges labeled either $\alpha$ or $\beta$ and vertices that are connected 
to $v$ by hyper-edges labeled either $\alpha$ or $\beta$. This is a hyper-graph of a 
spherical variety of the Levi subgroup of semi-simple rank 2 that corresponds to the roots 
$\alpha,\beta$ (see Lemma 3 of \cite{brion1999orbit}). It is enough to show the 
proposition for groups of semi-simple rank 2 because it will prove that defining $s$ using 
condition 2 and the Borel orbits in the hyper-edge labeled by $\alpha$ gives 
the same result as defining $s$ using condition 2 and the Borel orbits in the 
hyper-edge labeled by $\beta$.

    Thus, it is enough to check the proposition for all spherical varieties of 
groups of semi-simple rank 2. In order to minimize the case by case examination we reduce to
 symmetric spaces. Later, in Proposition \ref{affine support function} 
 we prove an affine version of such a reduction. The same method can be used in 
our current setting to prove the same reduction, switching the affine building for the 
spherical one. The idea is that if we can not reduce to a symmetric space then
we can reduce to a hyper-graph like in Proposition \ref{explicit description of hyper-graphs}.

    We further reduce to adjoint groups. 

    Let $\mathbf{Z}$ be the center of $\mathbf{G}$ and let 
$\mathbf{G}^{adj}=\mathbf{G}/\mathbf{Z}$ be the adjoint form of $\mathbf{G}$. Let 
$\pi_{adj}:\mathbf{G}\rightarrow \mathbf{G}^{adj}$ be the natural projection on 
the level of algebraic groups. The involution $\sigma$ acts on 
$\mathbf{G}^{adj}$. The fixed points of $\sigma$ on $\mathbf{G}^{adj}$ might be 
larger than $\pi_{adj}(\mathbf{H})$. 

Let $G^{adj}=\mathbf{G}^{adj}(\Bar{F})$ and $Z=\mathbf{Z}(\Bar{F})$.

We have $\Gamma_{\Bar{F}}(G/H)=\Gamma_{\Bar{F}}(G^{adj}/\pi_{adj}(H))$ and we want to compare it to
    $\Gamma_{\Bar{F}}(G^{adj}/(G^{adj})^\sigma)$. 

    We claim that $\pi_{adj}(H)$ contains the connected component 
of $1$ in $({G}^{adj})^\sigma$. 
    
    Any $g\in G^{adj}$ can be lifted to an element of $g'\in G'=[G,G]$, the derived 
group of $G$. The group $G'\cap Z$ is finite (see Subsection 16.2.5 in \cite{LAG})
 and if $g\in ( G^{adj})^\sigma$ then $g'\sigma(g')^{-1}\in G'\cap Z$. 
 Thus, for $g$ in the connected component of $1$ inside $( G^{adj})^\sigma$ 
 we can find $g'\in G'$ such that $g'\sigma(g')^{-1}=1$.

    To finish the reduction to adjoint groups it is enough to show that
    establishing the result for $\Gamma_{\Bar{F}}(G/H)$ is equivalent to establishing
     it for $\Gamma_{\Bar{F}}(G/H^0)$. Here $H^0$ is the connected component of $1\in H$.  
    
    The group $N_G(H^0)/H^0$ acts on $\Gamma_{\Bar{F}}(G/H^0)$ and it contains $H/H^0$. 
The hyper-graph $\Gamma_{\Bar{F}}(G/H)$ is the quotient of $\Gamma_{\Bar{F}}(G/H^0)$ 
by the action of $H/H^0$. 
Thus, if we have a function $s$ on $\Gamma_{\Bar{F}}(G/H^0)$ that satisfies the necessary conditions then
 the function on $\Gamma_{\Bar{F}}(G/H)$ obtained by summing $s$ over $H/H^0$ orbits also satisfies the conditions. 
 In the opposite direction, we can pullback the function $s$ on $\Gamma_{\Bar{F}}(G/H)$ and 
 divide its value on each orbit by the size of its stabilizer in $H/H^0$. This way we obtain 
 a function on $\Gamma_{\Bar{F}}(G/H^0)$ that satisfies the necessary conditions.

    Thus, it is enough to verify the result for the connected components of symmetric subgroups
    inside the groups $PGL_2\times PGL_2,PGL_3,PSp_4$ and $G_2$.

    \textbf{The case of $PGL_2\times PGL_2$:} The symmetric subgroups of 
$PGL_2\times PGL_2$ are either conjugate to the diagonal $PGL_2$ or equal to a 
product of two symmetric subgroups of $PGL_2$. In the former case the 
hyper-graph $\Gamma_{\Bar{F}}(X)$ has two vertices that are connected by two 
edges labeled $\alpha$ and $\beta$. In the latter case $\Gamma_{\Bar{F}}(X)$ is 
a product of two hyper-graphs of $PGL_2$. In any case it is easy to see that 
the proposition holds.

    \textbf{The case of $PGL_3$:} Up to isomorphism the involutions on $PGL_3$ 
are $x\rightarrow (x^t)^{-1}$ and $Ad(t)$ where $t$ is a diagonal matrix with 
two $-1$ and one $1$ on the diagonal. The connected components of their fixed 
points are $PO(3)$ and $PGL_2$. The hyper-graphs that correspond to these 
groups can be see in Figure \ref{fig:PGL_3}. 
In both cases there is a unique full closed orbit and $s$ is well defined. 

\begin{figure}
    \centering
\begin{tikzpicture}
\node[vertex] (v1) {};
\node[vertex,below left of=v1] (v2) {};
\node[vertex,below right of=v1] (v3) {};
\node[vertex,below right of=v2] (v4) {};

\begin{pgfonlayer}{background}
\draw[edge,color=yellow] (v1) -- (v2);
\begin{scope}[transparency group,opacity=.5]
\draw[edge,opacity=1,color=green] (v1) -- (v3);
\end{scope}
\draw[edge,color=green] (v2) -- (v4);
\draw[edge,color=yellow] (v3) -- (v4);
\end{pgfonlayer}

\node[elabel,color=yellow,label=left:\(\alpha\)]  (e1) at (-3,0) {};
\node[elabel,below of=e1,color=green,label=left:\(\beta\)]  (e2) {};

\end{tikzpicture}
    \begin{tikzpicture}
\node[vertex] (v1) {};
\node[vertex,below left of=v1] (v2) {};
\node[vertex,below right of=v1] (v3) {};
\node[vertex,below right of=v2] (v4) {};
\node[vertex,below left of=v2] (v5) {};
\node[vertex,below right of=v3] (v6) {};

\begin{pgfonlayer}{background}
\draw[edge,color=yellow] (v1) -- (v2);
\begin{scope}[transparency group,opacity=.5]
\draw[edge,opacity=1,color=green] (v1) -- (v3);
\end{scope}
\draw[edge,color=green] (v2) -- (v4) -- (v5) -- (v2);
\draw[edge,color=yellow] (v5) -- (v5);
\draw[edge,color=green] (v6) -- (v6);

\draw[edge,color=yellow] (v3) -- (v4) -- (v6) -- (v3);
\end{pgfonlayer}

\node[elabel,color=yellow,label=left:\(\alpha\)]  (e1) at (-3,0) {};
\node[elabel,below of=e1,color=green,label=left:\(\beta\)]  (e2) {};

\end{tikzpicture}
    \caption{On the left the hyper-graph $\Gamma_{\Bar{F}}(PGL_3/PO_3)$ and on 
the right the hyper-graph $\Gamma_{\Bar{F}}(PGL_3/PGL_2)$. $\alpha,\beta$ are 
the two simple roots of $PGL_3$}
    \label{fig:PGL_3}.
\end{figure}
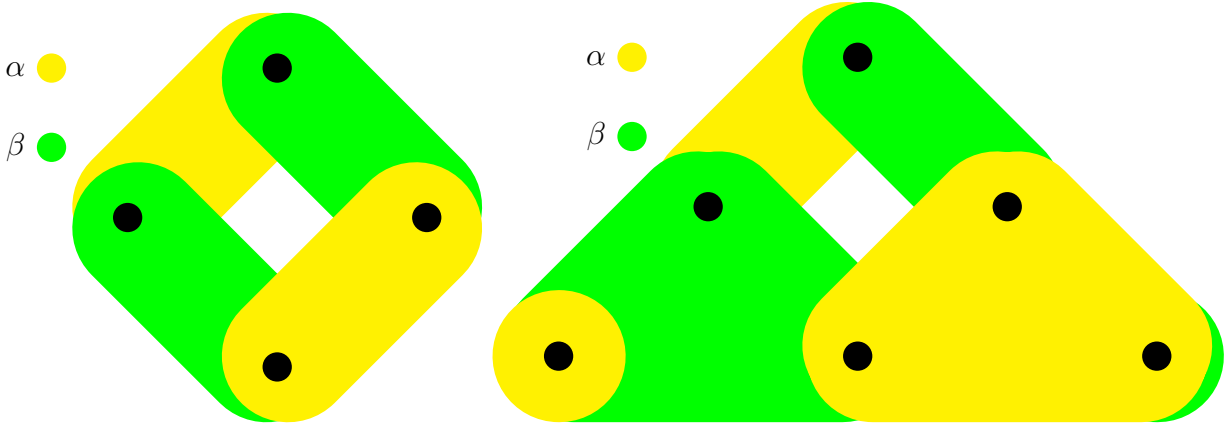

    \textbf{The case of $PSp_4$:} In this case, up to isomorphism there is a 
unique involution which is an inner involution. The connected component of
 its fixed points is isomorphic to $GL_2$ and its 
$\Gamma_{\Bar{F}}(X)$ hyper-graph is given in Figure \ref{fig:PSp_4}. It  also appears in 
Example 3 in \cite{brion1999orbit}. It is easy to check that the function $s$ is well defined.

\begin{figure}[ht]
    \centering
    
    \begin{tikzpicture}
\node[vertex] (v1) {};
\node[vertex,below left of=v1] (v2) {};
\node[vertex,below right of=v1] (v3) {};
\node[vertex,below of=v1] (v4) {};
\node[vertex,below of=v2] (v5) {};
\node[vertex,below of=v3] (v6) {};
\node[vertex,below of=v4] (v7) {};
\node[vertex,below left of=v5] (v8) {};
\node[vertex,below left of=v7] (v9) {};
\node[vertex,below right of=v7] (v10) {};
\node[vertex,below right of=v6] (v11) {};

\begin{pgfonlayer}{background}
\draw[edge,color=green] (v1) -- (v2) -- (v3) -- (v1);
\begin{scope}[transparency group,opacity=.5]
\draw[edge,opacity=1,color=yellow] (v1) -- (v4);
\end{scope}
\draw[edge,color=yellow] (v2) -- (v5);
\draw[edge,color=yellow] (v3) -- (v6);
\draw[edge,color=green] (v4) -- (v7);

\draw[edge,color=green] (v5) -- (v8) -- (v9) -- (v5);
\draw[edge,color=green] (v6) -- (v10) -- (v11) -- (v6);
\draw[edge,color=yellow] (v7) -- (v9) -- (v10) -- (v7);
\draw[edge,color=yellow] (v8) -- (v8);
\draw[edge,color=yellow] (v11) -- (v11);

\end{pgfonlayer}

\node[elabel,color=yellow,label=left:\(\alpha\)]  (e1) at (-3,0) {};
\node[elabel,below of=e1,color=green,label=left:\(\beta\)]  (e2) {};

\end{tikzpicture}
    \caption{The hyper-graph $\Gamma_{\Bar{F}}(PSp_4/GL_2)$. $\alpha,\beta$ 
are the two simple roots of $PSp_4$.}
    \label{fig:PSp_4}
\end{figure}
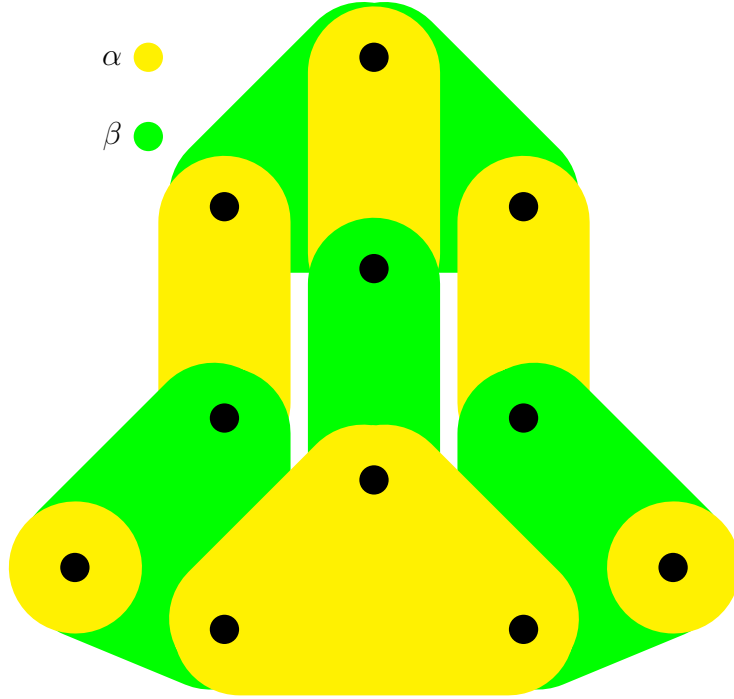

    \textbf{The case of $G_2$:} In this case, up to isomorphism there is a 
unique involution which is an inner involution. Its fixed points group 
is isomorphic to $PGL_2\times SL_2$. Its hyper-graph is given in Figure 
\ref{fig:G_2}. For this hyper-graph the function $s$ is well defined.
    
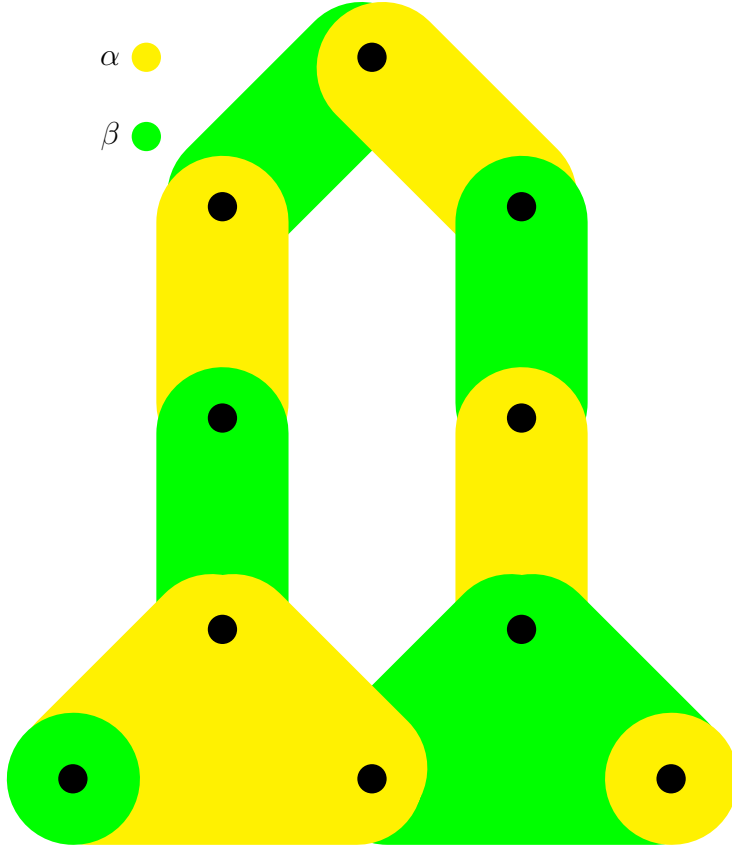
\begin{figure}[ht]
    \centering
    \begin{tikzpicture}
\node[vertex] (v1) {};
\node[vertex,below left of=v1] (v2) {};
\node[vertex,below right of=v1] (v3) {};
\node[vertex,below of=v2] (v4) {};
\node[vertex,below of=v3] (v5) {};
\node[vertex,below of=v4] (v6) {};
\node[vertex,below of=v5] (v7) {};
\node[vertex,below right of=v7] (v10) {};
\node[vertex,below right of=v6] (v11) {};
\node[vertex,below left of=v6] (v12) {};

\begin{pgfonlayer}{background}
\draw[edge,color=green] (v1) -- (v2);
\begin{scope}[transparency group,opacity=.5]
\draw[edge,opacity=1,color=yellow] (v1) -- (v3);
\end{scope}
\draw[edge,color=yellow] (v2) -- (v4);
\draw[edge,color=green] (v3) -- (v5);
\draw[edge,color=green] (v4) -- (v6);
\draw[edge,color=yellow] (v5) -- (v7);
\draw[edge,color=green] (v7) -- (v10)-- (v11) -- (v7);
\draw[edge,color=yellow] (v6) -- (v12)-- (v11) -- (v6);

\draw[edge,color=yellow] (v10) -- (v10);
\draw[edge,color=green] (v12) -- (v12);

\end{pgfonlayer}

\node[elabel,color=yellow,label=left:\(\alpha\)]  (e1) at (-3,0) {};
\node[elabel,below of=e1,color=green,label=left:\(\beta\)]  (e2) {};

\end{tikzpicture}

    \caption{The hyper-graph $\Gamma_{\Bar{F}}(G_2/PGL_2\times SL_2)$. 
$\alpha,\beta$ are the two simple roots of $G_2$.}
    \label{fig:G_2}
\end{figure}
    
\end{proof}

\end{section}

\begin{section}{Harmonic functions in the finite field case}\label{s6}
    In this Section we construct a non-zero harmonic function on 
$\Gamma_{F}(X)$ in the case of $F$ being a finite field, $X$ being quasi-split 
and $G,H$ both split over $F$. Our main tools are Proposition 
\ref{dimesnion} of the previous section and the following result of 
\cite{MR4328090}.

    \begin{theorem}[\cite{MR4328090}]\label{Hunnel}
        Assume that both $G$ and $H$ are split over $F$, the natural map 
$\pi:\Gamma_{F}(X)\rightarrow \Gamma_{\Bar{F}}(X)$ is surjective.
    \end{theorem}

    \begin{Remark}
        A result related to Theorem \ref{Hunnel} for general spherical varieties appears in Proposition 2.3.1
        of \cite{sakellaridis2021rank1}. 
    \end{Remark}

\begin{defn}
    Define a partial order on $\Gamma_{F}(X)$ by taking the minimal partial 
order such that for every two Borel orbits $x_1,x_2$ in the same hyper-edge,  
$x_1< x_2$ if $\overline{x_1}\subsetneq \overline{x_2}$. 
We refer to this order as the Bruhat order over $F$.
\end{defn}

    \begin{prop}\label{iso}
        Let $v$ be a dense $B$ orbit of $X$ over $F$. Consider 
$\Gamma^v_{F}(X)$, the sub hyper-graph of $\Gamma_{F}(X)$ consisting of all 
vertices smaller than $v$ in the Bruhat order. The natural map $\pi:\Gamma^v_{F}(X)\rightarrow 
\Gamma_{\Bar{F}}(X)$ is an isomorphism.
    \end{prop}

    Before proving that we need another result.

    \begin{prop}\label{An orbit above exists}
        Let $x\in \Gamma^v_{F}(X)$ and let $\alpha\in \Delta$ be a simple root such that 
        the hyper-edge labeled by $\alpha$ that contains $\pi(x)$ also contains an orbit in
  $\Gamma_{\Bar{F}}(X)$ larger than $\pi(x)$. We claim that we
  can find $y\in \Gamma^v_{F}(X)$ larger than $x$ 
  such that $y$ and $x$ are in a hyper-edge labeled by $\alpha$. 
    \end{prop}

    \begin{proof}
        We prove this by induction on the Bruhat order.

  We know that there there is a simple root $\beta$ 
  such that the hyper-edge labeled by $\beta$ that contains 
  $x$ also contains an orbit in $\Gamma^v_{F}(X)$ larger than $x$.

  Consider $\Gamma_{\alpha,\beta}$, the sub hyper-graph of $\Gamma_{F}(X)$ containing
  hyper-edges labeled by either $\alpha$ or $\beta$ and all the vertices that are connected to 
  $x$ by such hyper-edges.
  Like in the proof of Proposition \ref{support function}, $\Gamma_{\alpha,\beta}$ is 
  the hyper-graph of a spherical variety $X'$ of a group of semi-simple rank 2. Either $\Gamma_{\alpha,\beta}$
  is one of the hyper-graphs described by Proposition \ref{explicit description of hyper-graphs}
  or $X'$ is a symmetric space and $\Gamma_{\Bar{F}}(X')$ is one of the hyper-graphs described
   in the proof of Proposition \ref{support function}.
  In both cases, it is not difficult to check that every maximal vertex of $\Gamma_{\alpha,\beta}$
  that sits above $x$ also sits above some vertex larger than $x$ in the
  hyper-edge labeled by $\alpha$ that contains $x$. By induction we know that one of these maximal vertices 
  is in $\Gamma^v_{F}(X)$, thus we can find $y$ as desired.
    \end{proof}

    \begin{proof}{Proof of Proposition \ref{iso}}
        We prove the result by induction on the Bruhat order. In both hyper-graphs there
        is a unique maximal orbit, this is the basis of the induction. 
         
        Assume we have an orbit $x\in \Gamma^v_{F}(X)$ and we 
have a simple root $\alpha$ and $y\in \Gamma_{\Bar{F}}(X)$ such that $y$ and 
$\pi(x)$ are in a hyper-edge labeled by $\alpha$. For surjectivity, it is enough to show that 
there is an orbit $y'\in \Gamma^v_{F}(X)$ such that $y',x$ are in a 
hyper-edge labeled by $\alpha$ and $\pi(y')=y$.

Let $g\in G$ such that $x=BgH$ and let $T$ be a maximal split torus of $G$ 
such that $g^{-1}Tg$ is $\sigma$ stable. Consider $\alpha:T\rightarrow \mathbb{G}_m$ as a root of $T$.
Let $T_\alpha=Ker(\alpha)$ ane let $M_\alpha=Z_G(g^{-1}T_\alpha g)$.
The Borel orbits on the spherical variety $M_\alpha/M_\alpha\cap H$
 correspond to the vertices in the hyper-edge labeled by $\alpha$ that contains $x$.

Corollary 5.17 of \cite{MR4328090} says that the map $\pi$ restricted to $M_\alpha$
is surjective. Thus, we can find $y$ as desired and surjectivity follows.

Now, we prove injectivity of $\pi$. Let $x_1,x_2\in \Gamma^v_{F}(X)$
 and assume that $\pi(x_1)=\pi(x_2)$.
We can assume that neither of $x_1,x_2$ is the maximal orbit in $\Gamma^v_{F}(X)$.
 Let $\alpha\in \Delta$
be a simple root such that the hyper-edge labeled by $\alpha$ that contains $\pi(x_1)$ also
contains an orbit in $\Gamma_{\Bar{F}}(X)$ larger than $\pi(x_1)$. 

By Proposition \ref{An orbit above exists} we can find $y_1,y_2\in \Gamma^v_{F}(X)$
 larger than $x_1,x_2$ respectively, such that $y_1$ and $x_1$ are 
in a hyper-edge labeled by $\alpha$ and so are $y_2$ and $x_2$. 

  In this situation we have $\pi(y_1)=\pi(y_2)$, by induction $y_1=y_2$. 
  Thus, the hyper-edge containing $x_1$ and $y_1$ is the same hyper-edge
    containing $y_2$ and $x_2$.

It is enough to show that the restriction of $\pi$ to a single hyper-edge is injective on the 
set of orbits of minimal dimension in the hyper-edge. This holds because the vertices in the
hyper-edge correspond to Borel orbits on a spherical variety of $PGL_2$. 
And for any such variety passing to the algebraic closer is injective on closed 
orbits (see \cite{knop} for the description of these varieties and their Borel 
orbits).

     \end{proof}

\begin{defn}
    Let $x$ be a vertex of $\Gamma_{F}(X)$, denote by $m(x)$ the number of 
dense orbits above it in the Bruhat order.

    For any $\phi:\Gamma_{\Bar{F}}(X)\rightarrow\mathbb{C}$, we define 
$\pi^!(\phi):\Gamma_{F}(X)\rightarrow\mathbb{C}$ by 
$\pi^!(\phi)(x)=m(x)\phi(\pi(x))$.
\end{defn}

    \begin{prop}\label{pull back harmonic}
        For $\phi\in \mathcal{H}(\Gamma_{\Bar{F}}(X))$ we have $\pi^!(\phi)\in 
\mathcal{H}(\Gamma_{F}(X))$.
    \end{prop}

    \begin{proof}
    Let $B$ be a Borel subgroup of $G$.
        Let $\s$ be a $B$ orbit on $X$ over $F$. Let $\alpha\in\Delta$ be a 
simple root and let $P_\alpha$ be the corresponding standard parabolic that contains 
$B$. Write $P_\alpha\s=\s_1\cup \s_1...\cup \s_k$ as a union of $B$ orbits. We 
want to show that $\sum^k_{i=1} m(\s_i)\phi(\pi(\s_i))=0$. It is enough to show 
that for every orbit $\s_l$, $\sum_{\pi(\s_i)=\pi(\s_l)} m(\s_i)$ does not 
depend on $\s_l$. From calculating what happens in the case of $G=PGL_2$, we 
know that if we restrict $\pi$ to $\s_1,...,\s_k$ then it is injective on the 
subset of orbits closed in $P_\alpha \s$ and it maps all orbits dense in 
$P_\alpha \s $ to a single orbit.
        
        Let $v$ be a dense $B$ orbit of $X$, that is bigger (in the Bruhat 
        order) than one of the orbits $\s_1,...,\s_k$.
         By Proposition \ref{iso} $\pi$ is an isomorphism when 
restricted to $\Gamma^v_F(X)$, thus $v$ is bigger than all the closed orbits in 
$P_\alpha\s$ and it is bigger than exactly one dense orbit in $P_\alpha\s$. 
Therefore, if $\s_l$ is a closed orbit in $P_\alpha\s$ then $v$ is counted once 
in $\sum_{\pi(\s_i)=\pi(\s_l)} m(\s_i)=m(\s_l)$. If $\s_l$ is dense in 
$P_\alpha \s$, then $v$ is also counted once in $\sum_{\pi(\s_i)=\pi(\s_l)} m(\s_i)$ 
as it contains only one of the dense orbits. Thus, in any case $v$ is counted 
once and therefore the sum does not depend on $\s_l$.  
    \end{proof}

    \begin{cor}\label{cor dimension closed orbits}
        The dimension $dim\mathcal{H}(\Gamma_{F}(X))$ is larger or equal to the 
number of full closed Borel orbits in $\Gamma_{\Bar{F}}(X)$.
    \end{cor}
    \begin{proof}
        This follows immediately from Propositions \ref{dimesnion} and \ref{pull back harmonic}. 
    \end{proof}

    \begin{Remark}\label{inequality}
        There are case where equality does not hold. For example in the case of
$X=SL_2/\mathbb{G}_m$. In this case $dim\mathcal{H}(\Gamma_{F}(X))=3$ but there 
are only two full closed orbits in $\Gamma_{\Bar{F}}(X)$. 
    \end{Remark}

    As a result we obtain the following:
    \begin{theorem}
        Let $F$ be a finite field, assume that both $G$ and $H$ are split over 
$F$, then $X$ is quasi-split if and only if $St$ is $H$ distinguished. 
    \end{theorem}

    \begin{proof}
        The fact that $X$ being distinguished is a necessary condition can be 
proven like in Section \ref{s4}. The other direction follows from 
Corollary \ref{cor dimension closed orbits}, Proposition \ref{springer}, and proposition \ref{harmonic finite}.
    \end{proof}
\end{section}

\begin{section}{Harmonic functions in the non archimedean local field 
case}\label{s7}
Let $F$ be a non archimedean local field. Let $\s$ be the ring of integers of $F$ 
and let $k$ be its residue field.

In this Section we construct a non-zero harmonic function on 
$\Gamma^0_{aff,F}(X)$ under the assumptions that both $G$ and $H$ are split over $\s$, 
$\sigma$ is defined over $\s$, and $X$ is quasi-split, excluding some specific 
cases.

Our construction proceeds in two independent steps. First, like in 
Proposition \ref{support function} we construct a function $s$, this time on 
$H$ orbits of chambers in $\B$ to $\Z[S]$. Here, $S$ is the 
set of full orbits of $l_\sigma$ length zero in $\Gamma^0_{aff,F}(X)$. Notice
that orbits of $l_\sigma$ length zero are just orbits of chambers $\C$ such that $\sigma(\C)=\C$.

In the second step, we show that under our assumptions there are full orbits of $l_\sigma$ length 
zero, excluding in some specific cases which we analyze separately. 

We begin with the first step.

\begin{defn}
    Recall that the vertices of $\Gamma^0_{aff,F}(X)$ can be identified with 
$H$ orbits of chambers in $\B$, and recall the length function $l_\sigma$ on 
them, defined as $l_\sigma(\C)=d(\C,\sigma(\C))$. Denote by $S$ the set of chambers of 
$l_\sigma$ length zero that are full (with respect to $\Tilde{\Delta}$).
\end{defn}

\begin{defn}
    We define an equivalence relation $\sim$ on $\Gamma^0_{aff,F}(X)$.
    It is the minimal equivalence relation such that if $v_1,v_2$ are two vertices in the same
    hyper-edge of $\Gamma^0_{aff,F}(X)$ and $l_\sigma(v_1)=l_\sigma(v_2)$ is the
    maximal $l_\sigma$ length in this hyper-edge then $v_1\sim v_2$. 
    And if $v_1\sim v_2$ then for every simple reflection $\alpha\in \Tilde{\Delta}$ the
    hyper-edges labeled by $\alpha$ that contain $v_1$ and $v_2$ are equivalent.
    
    If two orbits are $\sim$ equivalent then they are identified over the algebraic closure of $F$. 
    
\end{defn}

\begin{prop}\label{affine support function}
    There exists a unique function $s:\Gamma^0_{aff,F}(X)\rightarrow \Z[S]$ such that:
    \begin{enumerate}
        \item For any $x\in S$ we have $s(x)=x$.
        \item For any $E$, a hyper-edge of $\Gamma^0_{aff,F}(X)$, we have $\sum_{v\in 
E}s(v)=0$.
        \item If $v_1\sim v_2$ then $s(v_1)=s(v_2)$.
    \end{enumerate}
\end{prop}

\begin{proof}
    Uniqueness is clear by induction on the $l_\sigma$ length.

    We construct $s$ by induction on the $l_\sigma$ length of a chamber. On full 
chambers of $l_\sigma$ length zero it is defined by the first property, on non full 
chambers of $l_\sigma$ length zero we define $s$ to be zero. 

    Now, let $\C$ be a chamber in $\B$ and let 
$\alpha,\beta\in\Tilde{\Delta}$ be two simple reflections such that each of the 
hyper-edges labeled $\alpha,\beta$ that contains $\C$, does not contain a 
chamber whose $l_\sigma$ length is larger than $l_\sigma(\C)$. Like in the proof of 
Proposition \ref{support function} it is enough to verify the result for the 
maximal sub hyper-graph that contains $H\C$ and has hyper-edges only with 
labels $\alpha$ or $\beta$. We denote this hyper-graph by 
$\Gamma_{\alpha,\beta}$.

    Each of $\alpha,\beta$ corresponds to a codimension 1 facet of $\C$, denote 
them by $f_\alpha$ and $f_\beta$ respectively, we may assume $f_\alpha\neq f_\beta$. Consider 
$f_{\alpha,\beta}=f_\alpha\cap f_\beta$. 
    
    There a few options.
\begin{enumerate}
    \item $f_{\alpha,\beta}=\emptyset$.
    \item $f_{\alpha,\beta}\neq\emptyset$, and $f_{\alpha,\beta}$ does not
     span a $\sigma$ stable hyper plane.
    \item $f_{\alpha,\beta}\neq\emptyset$, and $f_{\alpha,\beta}$ spans a $\sigma$ stable hyper plane.
\end{enumerate}

We deal with each case separately. 

\begin{enumerate}
    \item  This case happen if and only if $f_\alpha$ and $f_\beta$ are parallel.
    
        Let $\C_\alpha$ and $\C_\beta$ be the chambers in $\A$ that are different from $\C$
        and contain $f_\alpha$ and $f_\beta$ respectively. By our assumption 
        $l_\sigma(\C_\alpha)\leq l_\sigma(\C)$ and also $l_\sigma(\C_\beta)\leq 
       l_\sigma(\C)$. A sharp inequity for one of the chambers contradicts the other inequity,
        because if moving in one direction decreases $l_\sigma$ then moving in the other 
        direction must increase it. 
        Thus we must have $l_\sigma(\C)=l_\sigma(\C_\alpha)=l_\sigma(\C_\beta)$. 
        
        This implies that the walls containing $f_\alpha$ and $f_\beta$ are $\sigma$ stable. 
        Either $\sigma$ switches the half planes on the two side of $f_\alpha$ or it preserves
        them. It must preserve them because it can not switch both the half planes on the two sides
        of $f_\alpha$ and the half planes on the two sides of $f_\beta$ as $f_\alpha,f_\beta$ are 
        parallel. Remark 7.10 of \cite{my} says that in this situation the hyper edge 
        labeled by $\alpha$ that contains $\C$ either contains a single vertex
        or contains a vertex with $l_\sigma$ length larger than $l_\sigma(\C)$.
        
        By our assumption it means that the hyper-graph $\Gamma_{\alpha,\beta}$
        contains a single vertex and there is nothing to prove.

            \item In this case $\Gamma_{\alpha,\beta}$ is described by Proposition 
    \ref{explicit description of hyper-graphs} below. Checking the result for
    these hyper-graphs is easy.

    \item In this case, $\Gamma_{\alpha,\beta}$, is of the form 
$\Gamma_{k}(X')$ for some symmetric space $X'$ over a group $G'$ of semi-simple rank 2 (see 
Section 6 of \cite{my}). We know that $\Gamma_{\Bar{k}}(X')$ is one of the 
hyper-graphs discussed in the proof of Proposition \ref{support function}. In particular 
we know the result for $\Gamma_{\Bar{k}}(X')$ and we need to verify it for $\Gamma_{k}(X')$.

All the Borel orbits of the space $X'$ are defined over $k$ because $H$ is split and all
$I$ orbits of $X$ are defined over $F$. Thus we can apply the same ideas of Section \ref{s6}.

The map $\pi:\Gamma_{k}(X')\rightarrow \Gamma_{\Bar{k}}(X')$ is surjective.
It is also injective on the set of closed orbits as the stabilizer of a
closed orbit is a Borel subgroup of $H$.

For every dense orbit $v$ in $\Gamma_{\Bar{k}}(X')$,
 the map $\pi$ is an isomorphism when restricted to $\Gamma^v_{k}(X')$.
 
 This allows us to pull the function $s$ on $\Gamma_{\Bar{k}}(X')$ back to $\Gamma_{k}(X')$ 
   (like in Proposition \ref{pull back harmonic}) and verify the result for $\Gamma_{k}(X')$. 
    \end{enumerate}
\end{proof}

\begin{prop}\label{explicit description of hyper-graphs}
    Let $\C$ be a chamber in $\B$ and let $\alpha,\beta\in \Tilde{\Delta}$ be two
     simple reflections such that each of the hyper-edges labeled $\alpha$ or $\beta$
    that contains $\C$, does not contain a chamber whose $l_\sigma$ length is 
    larger than $l_\sigma(\C)$. 
    
    Let $\Gamma_{\alpha,\beta}$ be the maximal sub hyper-graph of $\Gamma^0_{aff,F}(X)$ that contains 
    $H\C$ and has hyper-edges only with labels $\alpha$ or $\beta$.
    
    Let $f_\alpha$ and $f_\beta$ be the facets of 
    $\C$ that correspond to $\alpha$ and $\beta$ respectively. 
    Assume that $f_\alpha\neq f_\beta$.
    Also assume that $f_\alpha\cap f_\beta$ is non empty and does not
     span a $\sigma$ stable affine hyper plane.
    
    Then, $\Gamma_{\alpha,\beta}$ is either a path graph, a cycle graph on an even 
    number of vertices, or a hyper-graph of the form shown in Figure \ref{fig:box hyper-graph}.
    In the latter case, $\Gamma_{\alpha,\beta}$ contains exactly two edges of size larger
    than 2 and it can be split to two graphs which are either path graphs or 
    cycle graphs on an even number of vertices. 
    The vertices of minimal $l_\sigma$ length in $\Gamma_{\alpha,\beta}$
     are contained in one of the two hyper-edges of size larger than 2. 
\end{prop}

\begin{proof}
     Consider 
$f_{\alpha,\beta}=f_\alpha\cap f_\beta$. 
    Let $\A$ be a $\sigma$ stable apartment that contains $\C$ and let
     $W_{\alpha,\beta}$ be the set of all walls in $\A$ that contain 
$f_{\alpha,\beta}$.

There are two cases:
\begin{enumerate}
    \item $W_{\alpha,\beta}$ does not contain a $\sigma$ stable wall.
    \item $W_{\alpha,\beta}$ does contain a $\sigma$ stable wall.
\end{enumerate}

We deal with each case separately. 
\begin{enumerate}

\item  Let $\C'\subset \A$ be the reflection of $\C$ across 
$f_{\alpha,\beta}$. There are no $\sigma$ stable walls that contain 
$f_{\alpha,\beta}$, thus $\Gamma_{\alpha,\beta}$ is a graph (see Lemma 5.5 of 
\cite{my}). This graph is a cycle with an even number of vertices, the edges are
alternating between edges labeled by $\alpha$ and edges labeled by $\beta$. All the 
vertices are $H$ orbits of chambers in $\A$ that contain $f_{\alpha,\beta}$.

    \item In this case $W_{\alpha,\beta}$ can contain only a 
single wall of $\A$ that is $\sigma$ stable. Otherwise, their intersection 
is a $\sigma$ stable affine hyper plane spanned by $f_{\alpha,\beta}$.
    
Let $\C_\alpha$ and $\C_\beta$ be the chambers in $\A$ that are different from $\C$
and contain $f_\alpha$ and $f_\beta$ respectively.

    First, we claim that one of $f_\alpha,f_\beta$ must lay in the unique 
$\sigma$ stable wall in $W_{\alpha,\beta}$. 
Assume the contrary, then $l_\sigma(\C_\alpha)< l_\sigma(\C)$ and $l_\sigma(\C_\beta)< l_\sigma(\C)$. 
Therefore, for any point $x$ in the interior of $\C$, that is close enough to $f_\alpha$, the 
line connecting $x$ and $\sigma(x)$ must intersect $f_\alpha$. The same is true for $f_\beta$. 
    
    Thus, by considering points in $\C$ close enough to $f_{\alpha,\beta}$ we 
can find a codimension one hyper plane $P$ that pass through $f_{\alpha,\beta}$ and 
intersects the interior of $\C$. This hyper plane is not a wall of $\A$ and thus 
is not in $W_{\alpha,\beta}$. The intersection of $P$ with the $\sigma$ stable wall in 
$W_{\alpha,\beta}$ is a $\sigma$ stable affine hyper plane spanned by $f_{\alpha,\beta}$. 
This is a contradiction.

    Without loss of generality we assume that $f_\alpha$ is contained in the unique 
$\sigma$ stable wall in $W_{\alpha,\beta}$, denote this wall by $W_\alpha$. 
Let $E_{\C,\alpha}$ be the hyper-edge labeled by $\alpha$ 
that contains $\C$. There are two options, either $E_{\C,\alpha}$ contains 
some chamber whose $l_\sigma$ length is smaller than $\C$ or it does not.

    Let $\C'\subset \A$ be the reflection of $\C$ across $f_{\alpha,\beta}$. 
Let $E_{\C'}$ be the hyper-edge that contains $\C'$ and is labeled by the simple 
reflection that corresponds to the facet of $\C'$ in $W_\alpha$.

    We prove that $E_{\C,\alpha}$ contains a chamber whose $l_\sigma$ length is 
less than $l_\sigma(\C)$ if and only if $E_{\C'}$ contains a chamber whose 
$l_\sigma$ length is less than $l_\sigma(\C')$. First lets see how this implies the result.

    \begin{enumerate}
        \item If $E_{\C,\alpha}$ does not contain a chamber whose $l_\sigma$ length is 
less than $l_\sigma(\C)$ then $\Gamma_{\alpha,\beta}$ is a graph, which is 
either a path graph or a cycle on an even number of vertices
 depending on the number of different orbits in 
$E_{\C,\alpha}$, which is either 1 or 2. The edges of this graph alternate between
edges labeled by $\alpha$ and edges labeled by $\beta$. 

        \item If $E_{\C,\alpha}$ does contain a chamber whose $l_\sigma$ length is 
less than $l_\sigma(\C)$ then there are four options for how 
$\Gamma_{\alpha,\beta}$ looks like, according to the number of elements in 
$E_{\C,\alpha}$. By Theorem 7.13 of \cite{my}, the hyper-edge  $E_{\C,\alpha}$ 
contains one or two elements of minimal $l_\sigma$ length and one or two 
elements of maximal $l_\sigma$ length. The hyper-graph in the case 
$\#E_{\C,\alpha}=4$ can be seen in Figure \ref{fig:box hyper-graph}. 

In any case, there are two types of vertices, with different parity of $l_\sigma$ length.
Restricting to the vertices of each type gives a graph, which is either a path graph or a cycle on an even number of vertices. 
There are two hyper-edges of size larger than 2, $E_{\C,\alpha}$ and $E_{\C'}$. All
the vertices of minimal $l_\sigma$ length are contained in $E_{\C'}$.  

    \end{enumerate}
    \end{enumerate}

    \begin{figure}
        \centering
        \begin{tikzcd}
	&& {} \\
	&&& \bullet & \bullet & {...} & \bullet & \bullet \\
	&&& \bullet & \bullet & {...} & \bullet & \bullet \\
	\bullet & \bullet & {...} & \bullet & \bullet \\
	\bullet & \bullet & {...} & \bullet & \bullet
	\arrow[color={rgb,255:red,92;green,92;blue,214}, no head, from=2-4, 
to=2-5]
	\arrow[color={rgb,255:red,214;green,92;blue,92}, no head, from=2-4, 
to=3-4]
	\arrow[color={rgb,255:red,214;green,92;blue,92}, no head, from=2-6, 
to=2-5]
	\arrow[color={rgb,255:red,92;green,92;blue,214}, no head, from=2-7, 
to=2-6]
	\arrow[color={rgb,255:red,214;green,92;blue,92}, no head, from=2-8, 
to=2-7]
	\arrow[color={rgb,255:red,92;green,92;blue,214}, no head, from=3-4, 
to=3-5]
	\arrow[color={rgb,255:red,214;green,92;blue,92}, no head, from=3-5, 
to=3-6]
	\arrow[color={rgb,255:red,92;green,92;blue,214},no head, from=3-6, 
to=3-7]
	\arrow[color={rgb,255:red,214;green,92;blue,92}, no head, from=3-7, 
to=3-8]
	\arrow[color={rgb,255:red,92;green,92;blue,214}, no head, from=3-8, 
to=2-8]
	\arrow[color={rgb,255:red,214;green,92;blue,92}, no head, from=4-1, 
to=2-4]
	\arrow[color={rgb,255:red,92;green,92;blue,214}, no head, from=4-1, 
to=4-2]
	\arrow[color={rgb,255:red,214;green,92;blue,92}, no head, from=4-2, 
to=4-3]
	\arrow[color={rgb,255:red,92;green,92;blue,214}, no head, from=4-3, 
to=4-4]
	\arrow[color={rgb,255:red,214;green,92;blue,92}, no head, from=4-4, 
to=4-5]
	\arrow[color={rgb,255:red,92;green,92;blue,214}, no head, from=4-5, 
to=2-8]
	\arrow[color={rgb,255:red,214;green,92;blue,92}, no head, from=5-1, 
to=3-4]
	\arrow[color={rgb,255:red,214;green,92;blue,92}, no head, from=5-1, 
to=4-1]
	\arrow[color={rgb,255:red,92;green,92;blue,214}, no head, from=5-1, 
to=5-2]
	\arrow[color={rgb,255:red,214;green,92;blue,92}, no head, from=5-2, 
to=5-3]
	\arrow[color={rgb,255:red,92;green,92;blue,214}, no head, from=5-3, 
to=5-4]
	\arrow[color={rgb,255:red,214;green,92;blue,92}, no head, from=5-4, 
to=5-5]
	\arrow[color={rgb,255:red,92;green,92;blue,214}, no head, from=5-5, 
to=3-8]
	\arrow[color={rgb,255:red,92;green,92;blue,214}, no head, from=5-5, 
to=4-5]
\end{tikzcd}\

        \caption{Box hyper-graph, the blue edges correspond to one simple root 
and the red edges to the other. The four vertices connected by red edges on one 
side of the box are in the same hyper-edge. Also the vertices on the other side 
of the box connected by blue edges are in the same hyper-edge. All other 
hyper-edges are of size 2.}
        \label{fig:box hyper-graph}
    \end{figure}
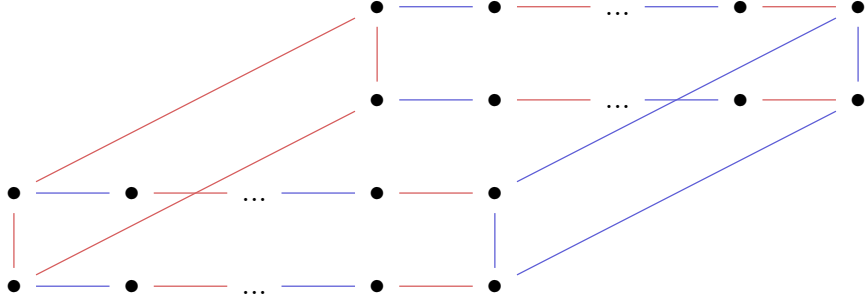
    
    Now we prove the assertion above, that is $E_{\C,\alpha}$ contains a chamber 
whose $l_\sigma$ length is less than $l_\sigma(\C)$ if and only if $E_{\C'}$ contains 
a chamber whose $l_\sigma$ length is less than $l_\sigma(\C')$. 
    
    The roles of $\C$ and $\C'$ are essentially symmetric, so we may assume 
that $E_{\C,\alpha}$ contains a chamber whose $l_\sigma$ length is less than 
$l_\sigma(\C)$ and show the same for $E_{\C'}$. Let $\C_m\in E_{\C,\alpha}$ be 
such that $l_\sigma(\C_m)<l_\sigma(\C)$. Let $\A'$ be a $\sigma$ stable 
apartment that contains $\C_m$. Remark 7.10 of \cite{my} implies that
$\sigma$ preserves the half spaces of $\A'$ defined by the affine hyper plane 
spanned by $f_{\alpha}$.

Let $\C_m'\subset \A'$ be the reflection of 
$\C_m$ across $f_{\alpha,\beta}$. We have $\C'_m\in E_{\C'}$ and 
$l_\sigma(\C'_m)<l_\sigma(\C')$ again by Remark 7.10 of \cite{my}.
\end{proof}

In order to construct a non-zero harmonic function, we only need to show that 
$S\neq\emptyset$. Before we discuss sufficient conditions for that, we one more result.

Recall that for a maximal split torus $T$ of $G$, and for every root $\alpha$ of $T$, 
we have a map $u_\alpha:\mathbb{G}_a\rightarrow \mathbf{G}$ that is an isomorphism
onto a one parameter subgroup $\mathbf{U}_\alpha$.
By Lemma 6.6 of \cite{Helminck_k_involutions} we may assume that these isomorphisms are
defined over $F$. We call such a collection $(u_\alpha)$ a realization of $G$ defined over $F$. 

\begin{prop}\label{edge of size 1}
    Let $\C$ be a $\sigma$ stable chamber inside a $\sigma$ stable apartment 
$\A$. Let $T$ be the maximal split torus that corresponds to $\A$.
 Let $f$ be a facet of $\C$ that lays in a wall 
defined by a $\sigma$ stable root $\alpha$. The group $H$ acts transitively on
all the chambers that contain $f$
 if and only if $\sigma(u_\alpha(x))=u_\alpha(x)$ for all $x\in F$.
\end{prop}

\begin{proof}
    For the first direction, assume that $\sigma(u_\alpha(x))=u_\alpha(x)$ for 
all $x\in F$. It implies that also $\sigma(u_{-\alpha(x)})=u_{-\alpha(x)}$ for 
all $x\in F$. Using the group generated by the elements of the form 
$u_{\pm\alpha(x)}$, the chamber $\C$ can be moved to any other chamber that 
contains $f$, thus they are all in the same $H$ orbit.

    For the second direction, we use the assumption that $\alpha$ is $\sigma$ 
stable. There are two options, either for all $x\in F$, 
$\sigma(u_\alpha(x))=u_\alpha(-x)$, or, for all $x\in F$, 
$\sigma(u_\alpha(x))=u_\alpha(x)$. We want to show that the second case 
happens. 

    Assume by contradiction that for all $x\in F$, 
$\sigma(u_\alpha(x))=u_\alpha(-x)$. 

Without loss of generality assume that there is $x\in F$ such that $u_\alpha(x)$ 
fixes $f$ but does not fix $\C$ (Otherwise replace $\alpha$ by $-\alpha$). 
By our assumption there is $h\in H$ such that $u_\alpha(x)\C=h\C$.

 Thus,
$u_\alpha(-x)\C=\sigma(u_\alpha(x)\C)=h\C=u_\alpha(x)\C$. Therefore, $u_\alpha(2x)$ 
fixes $\C$. We know that $u_\alpha(x)$ does not fix $\C$. This implies that the 
valuation of $2x$ is different from the valuation of $x$ which is a 
contradiction.
\end{proof}

Now, we give sufficient conditions for the existence of full orbits of $l_\sigma$ length zero.

\begin{prop}\label{not type A2n}
    Let $G$ be a simple group not of type $A_{2n}$, $\sigma$ an algebraic 
involution defined over $\s$ and assume that $H=G^\sigma$ is split over $\s$. 
Assume that $X=G/H$ is quasi-split. Then there is a chamber $\C$ in $\B$ with 
$l_\sigma(\C)=0$ such that the corresponding vertex in $\Gamma^0_{aff,F}(X)$ is 
full.
\end{prop}

\begin{proof}
    Consider $o$, the hyperspecial vertex of $\B$ that corresponds to 
$\mathbf{G}(\s)$. The vertex $o$ is $\sigma$ stable because $\sigma$ is defined 
over $\s$. The $\mathbf{H}(\s)$ orbits on the chambers that contain $o$ 
correspond to $\mathbf{H}(k)$ orbits on the flag variety of $\mathbf{G}(k)$. By 
Proposition \ref{springer} over $\Bar{k}$ there is a closed Borel orbit that 
is $\Delta$ full. By Theorem \ref{Hunnel} such an orbit exists also over $k$. 
This closed Borel orbit corresponds to a chamber $\C\subset \B$ that contains 
$o$ and is $\sigma$ stable, meaning $l_\sigma(\C)=0$. This chamber is $\Delta$ 
full also in $\Gamma^0_{aff,F}(X)$, we only need to show that $\C$ is
$\tilde{\Delta}$ full. 

        Let $T$ be a $\sigma$ stable maximal torus corresponding to a 
$\sigma$ stable apartment that contains $\C$. We may choose $\Delta$ to be a set of 
simple roots that are all positive on $\C$. Let $\gamma$ be the highest root 
with respect to this choice.

        Let $(u_\alpha)$ be a realization of $G$ defined over $F$.
        
        The set $\Tilde{\Delta}\backslash \Delta$ consists of a single element 
 corresponding to the unique vertex in the affine Dynkin diagram of $G$, 
that is not in the usual Dynkin diagram of $G$. The facet of $\C$ 
corresponding to $\Tilde{\Delta}\backslash \Delta$ is parallel to the walls 
defined by $\gamma$. The root $\gamma$ is $\sigma$ stable so by Proposition 
\ref{edge of size 1} it is enough to show that for all $x\in F$, 
$\sigma(u_\gamma(x))=u_\gamma(-x)$. 
        
         The chamber $\C$ is $\sigma$ stable, so $\sigma$ induces a map 
$\sigma:\Delta\rightarrow \Delta$.
         
        Section 7 of 
\cite{Springer1987TheCO} implies that there is an involution $\sigma_0:G\rightarrow G$ such that 
$\sigma_0(u_\alpha(x))=u_{\sigma(\alpha)}(x)$ for all $\alpha\in \Delta$ and 
$x\in F$. Furthermore, there is some $t\in T$ such that $\sigma=\sigma_0\circ 
Int(t)$, where $Int(t)$ is conjugation by $t\in T$. 
    
    For $\alpha\in \Delta$, such that $\sigma(\alpha)\neq \alpha$, denote $\beta=\sigma(\alpha)$.
     
    We have, for any $x\in F$, $\sigma_0\circ Int(t)\circ \sigma_0 \circ 
Int(t)(u_\alpha(x))=u_\alpha(x)$, therefore $\beta(t)\cdot \alpha(t)=1$.

    For $\alpha\in \Delta$, such that $\sigma(\alpha)=\alpha$, we have for any $x\in F$, 
$\sigma(u_\alpha(x))=u_\alpha(\alpha(t)x)$ so we must have $\alpha(t)^2=1$.  We 
know that $\C$ is $\Delta$ full, therefore by Proposition \ref{edge of size 1} 
we must have $\alpha(t)=-1$. 

    We can write $\gamma$ as a sum of simple roots 
$\gamma=\sum_{\alpha\in\Delta} n_\alpha\alpha$.
    Notice that as $\gamma$ is $\sigma$ stable we have 
$n_\alpha=n_{\sigma(\alpha)}$ for any $\alpha\in\Delta$.
    
    We have 
$$\gamma(t)=\Pi_{\alpha\in\Delta}\alpha(t)^{n_\alpha}=\Pi_{\alpha\in\Delta,\sigma(\alpha)=\alpha}\alpha(t)^{n_\alpha}=(-1)^{\sum_{\alpha\in\Delta,\sigma(\alpha)
=\alpha}n_\alpha}=(-1)^{\sum_{\alpha\in\Delta}n_\alpha}$$

    The sum $\sum_{\alpha\in\Delta}n_\alpha$ is equal to the Coxeter 
number of $G$ minus $1$. The Coxeter number of $G$ is even because $G$
 is not of type $A_{2n}$, thus $\gamma(t)=-1$.

       There are two cases, either $\sigma$ is an inner involution and then 
$\sigma_0$ is trivial or $\sigma$ is an outer involution.
    
    If $\sigma$ is an inner involution, then  
$\sigma(x_\gamma(a))=x_\gamma(\gamma(t)a)=x_\gamma(-a)$ and we are done.

    We are left with the case of $\sigma$ being an outer involution. In this case $G$ 
must be simply laced (see for example \cite{Springer1987TheCO}). It is enough 
to show that for any $x\in F$ we have $\sigma_0(u_\gamma(x))=u_\gamma(x)$. 

We say that a $\sigma$ stable positive root $\alpha$ is nice if 
for all $x\in F$ we have $\sigma_0(u_\alpha(x))=u_\alpha(x)$.
    
    We prove by induction that all $\sigma$ stable positive roots are nice.
    We already have this for simple roots by the construction of $\sigma_0$.
     It is enough to show the following:

    \begin{enumerate}
        \item If the sum of two nice roots $\alpha,\beta$ is a root, then $\alpha+\beta$ is nice.
        \item If $\alpha$ is nice, then for every simple root $\beta$ such that
         $\sigma(\beta)\neq\beta$ and such that $\alpha+\beta+\sigma(\beta)$ is a root,
          $\alpha+\beta+\sigma(\beta)$ is nice. 
    \end{enumerate}

    By Section 10.2 of \cite{Springer1981LinearAG} there is a bi-linear 
function $f:X^*(T)\times X^*(T)\rightarrow \Z$ such that if $\alpha,\beta$ are 
positive roots whose sum is a root we have for all $x,y\in F$, 
$$u_\alpha(x)u_\beta(y)u_\alpha(-x)u_\beta(-y)=u_{\alpha+\beta}((-1)^{f(\alpha,\beta)}xy)$$

    Therefore, if $\alpha,\beta$ are nice and $\alpha+\beta$ is a root then 
     $\alpha+\beta$ is nice.

    Let $\alpha$ be nice and let $\beta_1,\beta_2$ be two different simple roots 
such that $\sigma(\beta_1)=\beta_2$. Assume that $\alpha+\beta_1+\beta_2$ is 
also a root. We show that $\alpha+\beta_1+\beta_2$ and 
this will finish the proof. 
    
    By the construction of $\sigma_0$ we have for all $x\in F$, 
$\sigma_0(u_{\beta_1}(x))=u_{\beta_2}(x)$. 
    
    There is $c\in F^\times$ such that for all $x\in F$, 
$\sigma_0(u_{\alpha+\beta_1}(x))=u_{\alpha+\beta_2}(cx)$.  We have for all 
$x,y\in F$

$$\sigma_0(u_{\alpha+\beta_1}((-1)^{f(\alpha,\beta_1)}xy))=\sigma_0(u_\alpha(x)u
_{\beta_1}(y)u_\alpha(-x)u_{\beta_1}(-y))=$$

$$=u_\alpha(x)u_{\beta_2}(y)u_\alpha(-x)u_{\beta_2}(-y)=u_{\alpha+\beta_2}((-1)^
{f(\alpha,\beta_2)}xy)$$
    
    Therefore $c=(-1)^{f(\alpha,\beta_2)+f(\alpha,\beta_1)}$. We also have for 
all $x,y\in F$

$$\sigma_0(u_{\alpha+\beta_1+\beta_2}((-1)^{f(\alpha+\beta_1,\beta_2)}xy))=\sigma_0(u_{\alpha+\beta_1}(x)u_{\beta_2}(y)u_{\alpha+\beta_1}(-x)u_{\beta_2}(-y))=$$

$$=u_{\alpha+\beta_2}(cx)u_{\beta_1}(u)u_{\alpha+\beta_2}(-cx)u_{\beta_1}(-y)=u_
{\alpha+\beta_1+\beta_2}((-1)^{f(\alpha+\beta_2,\beta_1)}cxy)$$
    
    It is enough to show that 
$(-1)^{f(\alpha+\beta_1,\beta_2)}=(-1)^{f(\alpha+\beta_2,\beta_1)}c$. 

This is equivalent to  
$f(\alpha+\beta_1,\beta_2)+f(\alpha+\beta_2,\beta_1)+f(\alpha,\beta_2)+f(\alpha,
\beta_1)$ being even. Using the bi-linearity of $f$, it is enough to show that 
$f(\beta_1,\beta_2)+f(\beta_2,\beta_1)$ is even. By Subsection 10.2.3 of 
\cite{Springer1981LinearAG} it is equivalent to $\beta_1$ and $\beta_2$ being 
perpendicular. $G$ is not of type $A_{2n}$, so there are no involutions of its 
Dynkin diagram that send a vertex to a neighboring vertex. Therefore $\beta_1$ 
and $\beta_2$ must be perpendicular and we are done.
\end{proof}

We are left to deal with the case of $G$ being is of type $A_{2n}$.

\begin{prop}
    Let $G$ be a simple group split over $F$ of type $A_{2n}$, up to 
conjugation $G$ has two split symmetric subgroups, $H$ such that $G/H$ is quasi-split.
 Denote by $Z$ the center of $G$, $G/Z\cong PGL_{2n+1}$. The projection 
of $H$ to $G/Z$ is either conjugate to $P(GL_n\times GL_{n+1})$ or to 
$PO_{2n+1}$. In the first case $S=\emptyset$ and in the second case 
$S\neq\emptyset$.
\end{prop}

\begin{proof}
    The classification of quasi-split spaces is well known (see for example 
\cite{Springer1987TheCO}), As we assume that both $G$ and $H$ are split there
are no more cases over $F$ then in the algebraically closed field setting (see 
\cite{Helminck_k_involutions}). 

    In the case of $H/Z=P(GL_n\times GL_{n+1})(F)$ the involution $\sigma$ 
is an inner involution. By the same calculation as in the case of an inner
involution in the proof of 
Proposition \ref{not type A2n} we see that any vertex of 
$\Gamma_{aff,F}(X)$ of $l_\sigma$ length zero that $\Delta$ full is not $\tilde{\Delta}$ full.

    In the case of $H/Z$ being a split form of $PO_{2n+1}(F)$, the involution 
$\sigma$ is an outer involution. We can identify $X$ with the space of 
symmetric matrices in $PGL_{2n+1}(F)$ and the involution with $g\mapsto (g^{t})^{-1}$.
 Consider $g\in PGL_{2n+1}(F)$ represented by a matrix with ones on the anti diagonal
  and zeros everywhere else. The $I$ orbit of $g$ has $l_\sigma$ length zero orbit
  as $g^t=g=g^{-1}$.
   It is enough to check that it is full. It is 
easy to check that for every $s\in\Tilde{\Delta}$, $s^tgs\neq g$. Thus, the 
action of $s$ does not fix the orbit and therefore it is full and 
$S\neq\emptyset$.
\end{proof}

We know that the dimension of harmonic functions may be larger than $\#S$, so a 
priory $S=\emptyset$ does not guarantee that there are no non-zero harmonic 
functions on $\Gamma^0_{aff,F}(X)$. Nevertheless we have the following.

\begin{prop}
    Let $G=SL_{2n+1}(F)$, $H=(GL_n(F)\times GL_{n+1}(F))\cap SL_{2n+1}(F)$ and 
$X=G/H$. The only harmonic function on $\Gamma^0_{aff,F}(X)$ is zero.
\end{prop}

\begin{proof}
    We know that in this case $S=\emptyset$. It is enough to show that in any 
hyper-edge of $\Gamma^0_{aff,F}(X)$ there is only one vertex of maximal 
$l_\sigma$ length. After showing this we can prove by induction on $l_\sigma$ 
that any harmonic function must be zero on any chamber. The fact that there are 
no full vertices with $l_\sigma$ length zero is the base of the induction.

    Let $I_m$ be the identity matrix in $SL_m$.
    Let $\sigma$ be the involution given by conjugation by $-I_{n+1}\times 
I_n\in GL_{n+1}(F)\times GL_{n}(F)$ if $n$ is odd and by $I_{n+1}\times -I_n\in 
GL_{n+1}(F)\times GL_{n}(F)$ if $n$ is even.

    We use the description of hyper-edges given in Section 7 of \cite{my}. It is enough to show 
that for any $\sigma$ stable maximal split torus $T$ of $G$ and for every 
$\sigma$ stable root $\alpha:T\rightarrow F^\times$ of $T$, the restriction of $\alpha$ 
to the fixed points of $\sigma$, $T^\sigma$, is onto $F^\times$.

    All maximal split tori of $G$ are conjugate. We may switch $\sigma$ to another 
inner involution and assume that the torus $T$ is the torus of diagonal 
elements in $G$. Let $w\in N_G(T)$ be the element defining the inner 
involution. We denote the entries of $T$ by $t=(t_1,...,t_{2n+1})$. Write 
$w^{-1}tw=(t'_1,...,t'_{2n+1})$. Conjugation by $w$ is an involution and 
therefore for every $1\leq i \leq 2n+1$ we have either $t'_i=t_i$ or there is 
 $j\neq i$ such that $t'_i=t_j$ and $t'_j=t_i$.

    We have a $\sigma$ stable root $\alpha$. Recall that all the roots are of the 
form $\alpha_{i,j}(t_1,...,t_{2n+1})\rightarrow t_i/t_j$. Let $i,j$ be such that
$\alpha=\alpha_{i,j}$. The root $\alpha$ is 
$\sigma$ stable so we must have $t'_i=t_i$ and $t'_j=t_j$. There has to be some 
$k\neq i,j$ such that $t'_k=t_k$ because the number of indices switched by $w$ 
is even. For any $x\in F^\times$ take $t\in T$ with $t_i=x,t_k=x^{-1}$ and all 
the rest are 1. The element $t$ is fixed by the involution and 
$\alpha_{i,j}(t)=x$. Thus, the restriction of $\alpha$
 to the fixed points of the involution is onto $F^\times$.
\end{proof}

Now we can deduce the following theorem:

\begin{theorem}\label{Theorem chi_0}
    Let $G$ be a reductive group split over $\s$, let $H=G^\sigma$, and assume 
that $H$ is also split over $\s$.  Let $\mathbf{Z}=Z(\mathbf{G})$ be the center 
of $\mathbf{G}$.  $\mathbf{G}/\mathbf{Z}$ can be written as a product of 
adjoint simple groups $\mathbf{G}/\mathbf{Z}=G_1\times...\times G_m$. 
    Then there is a non-zero harmonic function on $\Gamma^0_{aff,F}(X)$ if and 
only if the following conditions hold:
    \begin{enumerate}
        \item $X=G/H$ is quasi-split.
        \item There is no $\sigma$ invariant $G_i$, such that $G_i\cong 
PGL_{2n+1}$ and $G^\sigma_i\cong P(GL_n\times GL_{n+1})$.
     \end{enumerate}
\end{theorem}

\begin{proof}
    For simple groups we already proved this.  Denote $Z=\mathbf{Z}(F)$,
     $G'=G/Z$ and $H'=H/Z$. We 
have, $\Gamma^0_{aff,F}(X)=\Gamma^0_{aff,F}(G'/H')$. The hyper-graph $\Gamma_{aff,F}(G'/H')$ 
can be written as a product of hyper-graphs of simple groups or of spaces of 
the from $K\times K/\Delta K$. We might get hyper-graphs of the form
$\Gamma_{aff,F}(G_i/H_i)$ where $G_i$ is simple but $H_i$ is smaller than the fixed points
of the involution on $G_i$. Like in the proof of Proposition \ref{support function}, $H_i$ must include
the connected component of the identity of the fixed points of the involution. The existence
of a Harmonic function on $\Gamma^0_{aff,F}(G_i/H_i)$ is equivalent to the existence of a 
harmonic function on $\Gamma^0_{aff,F}(G_i/G_i^\sigma)$.

Existence of a non-zero harmonic function on a 
product of hyper-graphs is equivalent to the existence of a non-zero harmonic 
function on each of them. Thus the results follows from the case of a simple 
group and the group case.

    In the group case it is easy to construct a non-zero harmonic function.
\end{proof}

Now we deduce the following:

\begin{cor}\label{main}
    Let $F$ be a non archimedean local field. Let $G$ and $H$ be split over 
$\s$, let $St_{\chi_0}$ be the Steinberg representation of $G$ that corresponds 
to the character $\chi_0$. $St_{\chi_0}$ is $H$ distinguished if and only if 
the conditions of Theorem \ref{Theorem chi_0} hold.
\end{cor}
\begin{proof}
    This follows from Theorem \ref{Theorem chi_0} and Proposition \ref{steinberg 
harmonic zero}.
\end{proof}
    
\end{section}

\begin{section}{The Steinberg representation of a general unramified 
character}\label{s8}

In the previous Section we answered the question when $St_{\chi_0}$ is $H$ 
distinguished. In this Section we answer the question when $St_\chi$ is 
$H$ distinguished for a general unramified character 
$\chi:G\rightarrow\mathbb{C}^\times$. 

Recall Proposition \ref{distinction harmonic general}, it gives us a condition for 
$St_\chi$ being $H$ distinguished. We need a way to check 
when there is a non-zero harmonic function $\phi$ on $\Gamma^1_{aff,F}(X)$
that satisfies the condition: for 
any chamber $\C$ and $h\in H$, 
$\phi(H^1h\C)=\chi^{-1}\chi_0(\omega(h^{-1}))\phi(H^1\C)$. Here, $H^1=H\cap 
Ker(\omega)$ is a normal subgroup of $H$.

We construct a subgroup $\Omega_H$ of $\Omega$ and show that if $S\neq\emptyset$ and 
 $\chi^{-1}\chi_0$ is trivial on $\Omega_H$ then there is
  a harmonic function on $\Gamma^1_{aff,F}(X)$ with the desired properties.

We need the following results.

\begin{prop}\label{maximal H torus}
    \begin{enumerate}
        \item Every $\sigma$ stable apartment that contains a chamber of 
$l_\sigma$ length 0 corresponds to a maximal split torus $T$ such that identity component
 $T^+=(T\cap H)^0$ is a maximal torus of $H$.
        \item Every two maximal split tori $T_1,T_2\subset G$ such that $T_1^+=(T_1\cap H)^0$ and 
$T_2^+=(T_2\cap H)^0$ are maximal tori of $H$ are $H$ conjugate.
    \end{enumerate}
\end{prop}

\begin{proof}
\begin{enumerate}
    \item Let $\A$ be a $\sigma$ stable apartment and let $T$ be the maximal 
split torus that corresponds to $\A$. Assume that $\C\subset \A$ is a chamber 
such that $l_\sigma(\C)=0$. Let $T^+=(T\cap H)^0$ be the identity component,
 assume by contradiction that $T^+$ is not a maximal torus of $H$.
  Let $M=C_G(T^+)$ be the centralizer of 
$T^+$. The group $M$ is a Levi of $G$ on which $\sigma$ acts.

Consider the apartment of the extended affine building of $M$ that corresponds to $T$.
Either this apartment contains a $\sigma$ stable wall, or all $\sigma$ stable tori of $M$ are $H$ conjugate.
This follows from the fact that there is a gallery connecting any two chambers and
Lemma 5.5 of \cite{my}.   

  We can construct such maximal tori of $M$ that are not $H$ conjugate. 
We know that $T^+$ is not a maximal torus of $H$, and so there is a maximal torus $T'\subset H$
that contains $T^+$, in particular $T'\subset M$. We can find a maximal torus
$T''\subset M$ that is $\sigma$ stable and $T'\subset T''$. The tori $T$ and $T''$
are maximal tori of $M$ and they are not $H$ conjugate as the dimensions of their intersection
with $H$ are different.

Thus, there is a $\sigma$ stable wall in the apartment of $M$ that corresponds to $T$.
The group $M$ is a Levi of $G$, so it implies that there is a $\sigma$ stable wall $W\subset\A$.
We know that $\sigma(\C)=\C$ so $\sigma$ must preserve the half planes 
defined by $W$. Let $\alpha$ be a root constant on $W$. We have $U_\alpha\subset M$
because the wall $W$ came from a wall of the apartment of $M$.

The translation defined by $\alpha^\vee$, which is the 
translation perpendicular to $W$, commutes with $\sigma$. By Proposition 5.7 of 
\cite{my}, we may assume that this translation is in $T^+$, this 
contradicts $U_\alpha\subset C_G(T^+)$.

    \item Let $T_1,T_2$ be two maximal split tori of $G$, such that $T_1^+=(T_1\cap H)^0$ and 
 $T_2^+=(T_2\cap H)^0$ are maximal tori of $H$. The tori $T_1^+$ and $T_2^+$ 
are $H$ conjugate, we may assume that $T_1^+=T_2^+$. Denote this torus by $T^+$.
 Consider $M=C_G(T^+)$, $\sigma$ acts on $M$ and 
every fixed point of $\sigma$ on $M$ is an element of $C_H(T^+)=T^+$. Clearly 
$T_1,T_2\subset M$. Any maximal split torus $T\subset M$ contains $T^+$ and 
thus its $\sigma$ split part $T^-=\{t\in T|\sigma(t)=t^{-1}\}^0$ is a maximal 
$\sigma$ split torus of $M$. By Proposition 10.3 of \cite{Helminck} and the fact that $M$ 
is split we get that the $\sigma$ split parts of $T_1$ and $T_2$ are $T^+=H\cap M$ 
conjugate. This can only happen if $T_1=T_2$.
\end{enumerate}
\end{proof}

In the rest of this section we assume that both $G$ and $H$ are split over $\s$. 
Let $0$ be the hyper special vertex of $\B$ that correspond to $G(\s)$. Let 
$\A$ be the $\sigma$ stable apartment of minimal rank that contains $0$. 
The $H\cap G(\s)$ orbits of chambers that contain $0$ correspond to 
Borel orbits on $G(k)/H(k)$ (see Section 6 of \cite{my}).
The group $H$ is split over $\s$ so all Borel orbits on $G/H$ are defined over $k$.
This implies that there is a $\sigma$ stable chamber $\C\subset \A$ that contains $0$.  

We need the following.

\begin{prop}\label{Omega H is well defined}
    Let $\C$ be a $\sigma$ stable chamber, then 
$\Omega_\C=\{\omega(h)|h\in H,h\C=\C\}\subset \Omega$ does not depend on $\C$.

We denote this subgroup by $\Omega_H$.
\end{prop}

\begin{proof}
    By Proposition \ref{maximal H torus} we may assume that $\C\subset \A$. 
Then, by Proposition \ref{same apartment} below we are done.
\end{proof}

\begin{prop}\label{same apartment}
    Let $\C_1,\C_2$ be two $\sigma$ stable chambers in $\A$, let $h\in H$ such  
that $h\C_2=\C_2$ then there is $h'\in H$ such that $h'\C_1=\C_1$ and 
$\omega(h')=\omega(h)$.
\end{prop}

\begin{proof}
    Let $T$ be the $\sigma$ stable maximal split torus corresponding to 
$\A$. We can lift the Weyl group $N_G(T)/T$ to a subgroup of $N_G(T)$ of 
element that fix $0\in\A$, denote this subgroup by $W$. 
    
First we show that any $\sigma$ stable chamber in $\A$ is in the $H$ orbit of a 
$\sigma$ stable chamber that contains $0$.

Let $\C\subset \A$ be $\sigma$ stable, we can find $t\in T$ and $\C_0\subset \A$ such that
$t\C_0=\C$ and $\C_0$ contains $0$. We have $t\C_0=\sigma(t\C_0)=\sigma(t)\sigma(\C_0)$,
equivalently $\sigma(t)^{-1}t\C_0=\sigma(\C_0)$. Both $\C_0$ and $\sigma(\C_0)$ contain $0$
and $\sigma(t)^{-1}t$ acts by translation on $\A$. This can only happen if $\sigma(t)^{-1}t$
acts trivially on $\A$. This means that the actions of $t$ and $\sigma$ on $\A$ commute.
By Proposition 5.7 of \cite{my} this implies that $t\in H$ and thus $\C_0$ is $\sigma$ stable
and $\C$ is in the same $H$ orbit as $\C_0$.

Therefore, it is enough to prove the result for $\C_1,\C_2$ two $\sigma$ stable
chambers that contain $0$. We can find $w_2\in W$ such that 
$w_2\C_1=\C_2$. Applying $\sigma$ we get that also $\sigma(w_2)\C_1=\C_2$, 
therefore we can find $t_2\in T^0$ such that $\sigma(w_2)=t_2w_2$. It implies that 
$\sigma(t_2)=t^{-1}_2$. We can assume that $t_2\in T^-$ is $1$ modulo the 
uniformaizer. 
Using Hensel's lemma we can find a square root of $t_2$ in $T^-$ 
(For more detail see the proof of Proposition 3.1 in \cite{my}). Denote this square root by 
$t^{\frac{1}{2}}_2$.

Let $h\in H$ such that $h\C_2=\C_2$. Consider 
$h'=w^{-1}_2t^{-\frac{1}{2}}_2ht^{\frac{1}{2}}_2w_2$, we claim that $h'\in H$. 
Computing we get that
    $\sigma(h')=w^{-1}_2t^{-1}_2t^{\frac{1}{2}}_2ht^{-\frac{1}{2}}_2t_2w_2=h'$. 
Clearly $\omega(h')=\omega(h)$ and $h'\C_1=\C_1$ so we are done.
\end{proof}

 We want to construct harmonic functions on $\Gamma_{aff,F}^1(X)$. By the same 
argument as in the proof of Proposition \ref{affine support function}, harmonic 
functions on $\Gamma_{aff,F}^1(X)$ can be constructed by specifying their 
values on the $H^1$ orbits on $\sigma$ stable chambers and extending 
using a function $s$ similar to the one constructed in Proposition \ref{affine 
support function}. There are finitely many $H^1$ orbits of $\sigma$ stable chambers.
 Denote the set of these orbits by $S'$ that are also full as
 vertices of $\Gamma_{aff,F}^1(X)$. Clearly $H$ acts on $S'$.

 \begin{prop}\label{harmonic function if charachter is good}
     Assume the following:
     \begin{enumerate}
        \item We have $S'\neq\emptyset$.
        \item For every $\sigma$ stable chamber $\C$ and $h\in H$ such that $h\C=\C$
           we have $\chi^{-1}\chi_0(\omega(h))=1$.
     \end{enumerate} 
     
     Then, there is a non-zero function $\phi\in\mathcal{H}(\Gamma^1_{aff,F}(X))$ 
such that for any chamber $\C$ and $h\in H$, 
$\phi(H^1h\C)=\chi^{-1}\chi_0(\omega(h^{-1}))\phi(H^1\C)$. 
 \end{prop}

 \begin{proof}
     We have a function $s:\Gamma^1_{aff,F}(X)\rightarrow \Z[S']$ like the one 
given by Proposition \ref{affine support function}. We 
define $\phi$ on $S'$ and then extend it to $\Gamma^1_{aff,F}(X)$. Choose 
$x\in S'$, for every $h\in H$ set $\phi(hx)=\chi^{-1}\chi_0(\omega(h^{-1}))$. This defines
$\phi$ on the $H$ orbit of $x$. We define $\phi$ to be zero on elements of $S'$ not in the 
$H$ orbit of $x$. The function $\phi$ is well defined on $S'$ by our assumptions. 

     Now extend $\phi$ to $\Z[S']$ and to $\Gamma^1_{aff,F}(X)$ by the function 
$s$.

     Clearly $\phi$ is harmonic and non-zero. We just have to check that for 
every $h\in H$ and every chamber $\C$, 
$\phi(H^1h\C)=\chi^{-1}\chi_0(\omega(h^{-1}))\phi(H^1\C)$.

     Notice that the map $s$ intertwines the action of $H$ on the vertices of 
     $\Gamma^1_{aff,F}(X)$ and the action of $H$ on $\Z[S']$.

     We have $s(H^1h\C)=hs(H^1\C)$. By the construction of $\phi$ and its values 
on $S'$ we have 
$\phi(H^1h\C)=\phi(hs(H^1\C))=\chi^{-1}\chi_0(\omega(h^{-1}))\phi(s(H^1\C))=\phi(H^1\C)$. 
 \end{proof}

 Now we deduce:

 \begin{theorem}\label{local field main}
     Let $G$ be a reductive group split over $\s$, let $H=G^\sigma$, and assume 
that $H$ is split over $\s$.  Let $\mathbf{Z}=Z(\mathbf{G})$ be the center of 
$\mathbf{G}$.  $\mathbf{G}/\mathbf{Z}$ can be written as a product of adjoint 
simple groups $\mathbf{G}/\mathbf{Z}=G_1\times...\times G_m$. 
     
     Let $\chi:G\rightarrow\mathbb{C}^\times$ be an unramified character, and 
let $St_\chi$ be the corresponding Steinberg representation.
    If the following conditions hold then $St_\chi$ is $H$-distinguished.
     \begin{enumerate}
         \item $X$ is quasi-split.
         \item There is no $\sigma$ invariant $G_i$, such that $G_i\cong 
PGL_{2n+1}$ and $G^\sigma_i\cong P(GL_n\times GL_{n+1})$.
         \item Let $\C$ be a chamber such that $l_\sigma(\C)=0$, 
$\chi^{-1}\chi_0$ is trivial on $\Omega_H$. 
     \end{enumerate}
 \end{theorem}

 \begin{proof}
    This follows from Propositions \ref{harmonic function if charachter is good}, \ref{distinction harmonic general}
    and (the proof of) Theorem \ref{Theorem chi_0}.

 \end{proof}
    
\end{section}

\begin{section}{Interpretation of the main results in terms of the dual 
group}\label{s9}
    Let $G^\vee$ be the Langlands dual of $G$. There are a few different 
definitions for $G_X^\vee$, the dual group of $X$. We follow the one given by 
Takeda for symmetric spaces in \cite{takeda}. In \cite{takeda} it is stated 
without a proof that their definition for $G_X^\vee$ is equivalent to the 
definition given in \cite{GaitsgoryNadler}. This definition is slightly
different from the one given in \cite{KS17}. In \cite{KS17} a map 
$\iota:SL_2\times G_X^\vee\rightarrow G^\vee$ is defined, this map is
in general not injective on $G_X^\vee$. The group defined in \cite{takeda}
is the image in $G^\vee$ of the group defined in \cite{KS17}.

Also in \cite{takeda} the map 
$\iota:SL_2\times G_X^\vee\rightarrow G^\vee$ is defined, it is injective 
on $G^\vee_X$. In this Section we identify the symmetric spaces $X$ such that the Langlands 
parameter of the Steinberg representation factors through $\iota$. 
        
    Let $\mathcal{U}$ be the variety of unipotent elements in $G^\vee$ and let 
$\mathcal{U}_X$ be the variety of unipotent elements in $G^\vee_X$.

    Let $W_F$ be the Weil-Deligne group of $F$. Let $\phi_{St}: W_F\times 
SL_2(\mathbb{C})\rightarrow G^\vee$ be the expected parameter of the  Steinberg 
representation under the expected local Langlands correspondence (see for 
example Section 1.4 of \cite{cunningham2024koszuldualitygeneralizedsteinberg}).

    Conjecture 16.5.1 of \cite{SV17}
implies that the distinction of the Steinberg representation should be detected 
by the factoring of $\phi_{St}$ through $\iota$. This condition can be stated
without assuming the Langlands correspondence as it is equivalent
to $\mathcal{U}_X$ intersecting the open $G^\vee$ orbit in $\mathcal{U}$.

In this section we prove the 
following:

    \begin{theorem}\label{theorem dual}
        Assume that $G$ and $H$ are split over $\s$. The representation
         $St_{\chi_0}$ is $H$ distinguished if and only if $\phi_{St}$ factors through $\iota$. 
    \end{theorem}

    \begin{Remark}
        Notice that in Theorem \ref{theorem dual}, the parameter is the 
parameter of $St=St_1$ and not of $St_{\chi_0}$
    \end{Remark}

    Before we present the proof we recall some properties of $G^\vee_X$ and 
$\iota$ from \cite{takeda}.

    \begin{subsection}{Properties of $G^\vee_X$}

    Let $T$ be a maximal $\sigma$ stable split torus of $G$ such that 
$T^-=\{t\in T|\sigma(t)=t^{-1}\}^0$ is a maximal split torus on which $\sigma$ 
acts as inverse.

    Let $R$ be the set of roots of $G$ with respect to $T$ and let $\Delta$ be 
a set of simple roots of $R$. The set $R$ can be chosen such that if $\alpha$ is a positive root 
and $\sigma(\alpha)\neq \alpha$ then $\sigma(\alpha)$ is negative (see Subsection
3.4 of \cite{takeda}). Let $\Delta_0\subset \Delta$ be the set of simple roots that are fixed by $\sigma$.

    \begin{prop}[Theorem 8.2 of \cite{takeda}]\label{iota defined}
        Let $2\rho_0=\sum_{\alpha\in\Delta_0}\alpha$.

    The image of $\iota$ on the torus of $SL_2(\mathbb{C})$ is the image of the 
$G^\vee$ coroot $2\rho_0:\mathbb{G}_m\rightarrow G^\vee$.
    \end{prop}

    \begin{prop}[Lemma 3.2 of \cite{takeda}]
    Assume that the root system of $G$ is irreducible. Let $w_0$ be the longest 
element of the Weyl group. Consider $\sigma^*:R\rightarrow R$ defined by 
$\sigma^*(\alpha)=-\sigma(w_0(\alpha))$.  Then $\sigma^*$ is an automorphism of 
the Dynkin diagram of $G$ such that for every 
$\alpha\in\Delta\backslash\Delta_0$ we have $\sigma(\alpha)+\sigma^*(\alpha)\in 
\Z\Delta_0$
        
    \end{prop}

    \begin{defn}[Definition 4.1 of \cite{takeda}]\label{diagram}
        Let $\alpha\in\Delta\backslash \Delta_0$, define 
$\Phi_\alpha=\Z(\Delta_0\cap\{\alpha,\sigma(\alpha)\})\cap R$. 

        Define the Dynkin diagram of $\alpha$ to be the Dynkin diagram of the 
irreducible component of $\Phi_\alpha$ that contains $\alpha$. We consider the 
Dynkin diagram with roots in $\Delta_0$ colored black and the action of 
$\sigma^*$ drawn on the diagram.

        These diagrams were classified in \cite{takeda} and appear in Table 1 
of \cite{takeda}.
    \end{defn}

    We are interested in diagrams with no black vertices. Such diagrams from 
Table 1 of \cite{takeda} appear in Table \ref{T:table}, the numbering is taken 
from \cite{takeda}. 

    \pgfkeys{/Dynkin diagram, edge-length=.7cm, root-radius=.07cm}

\begin{table}[h]
\caption{Rank one involutions with no black vertices}\label{T:table}
\bigskip{}
\centering
\renewcommand{\arraystretch}{1.5}
\begin{tabular}{cclcc}
\hline
No. & $\Phi$ & Diagram & Type 
\\\hline\hline
1 &$A_1\times A_1$  & \dynkin[labels={,}, involutions={1<above>[\sigma^*]2}, 
edge/.style={white}]A{oo} & 3   \\
2 & $A_1$ & \dynkin A{o} & 2 \\
5 & $A_2$ & \dynkin[involutions={1<above>[\sigma^*]2}]A{oo} & 1\\\hline

\end{tabular}
\end{table}

    There are four options for each simple root $\alpha\in\Delta$ (see Subsection 4.3 of 
\cite{takeda}):

    \begin{itemize}
        \item $\sigma(\alpha)=\alpha$, i.e. $\alpha\in\Delta_0$.
        \item $\alpha-\sigma(\alpha)$ is a root. These roots are called roots of type 1.
        \item $\frac{1}{2}(\alpha-\sigma(\alpha))$ is a root. These roots are called 
roots of type 2.
        \item None of the above holds. These roots are called roots of type 3. 
    \end{itemize}

    For a root $\alpha$ of type 3 there is a unique pair of 
    roots $\alpha_1,\alpha_2$ characterized by 
Proposition 5.2 of \cite{takeda}, one of the conditions is 
$\alpha-\sigma(\alpha)=\alpha_1+\alpha_2$, the second condition is that $\alpha_1,\alpha_2$ 
are strongly orthogonal.

    \begin{prop}[Subsection 5.3 of \cite{takeda}]\label{roots of G_X}
        Let $\alpha\in\Delta$ such that $\sigma(\alpha)\neq\alpha$. The group 
$G^\vee_X$ contains the following subgroups:
        \begin{enumerate}
            \item $\{u_{\alpha-\sigma(\alpha)}(x)|x\in\mathbb{C}\}$ if $\alpha$ 
is of type 1.
            \item 
$\{u_{\frac{1}{2}(\alpha-\sigma(\alpha))}(x)|x\in\mathbb{C}\}$ if $\alpha$ is 
of type 2.
            \item $\{u_{\alpha_1}(x)u_{\alpha_2}(x)|x\in\mathbb{C}\}$ if 
$\alpha$ is of type 3.
        \end{enumerate}
    \end{prop}

    \end{subsection}

    \begin{subsection}{Proof of Theorem \ref{theorem dual}}
        
    We recall how $\phi_{St}$ is expected to behave, see for 
example Lemma 1.8 in \cite{cunningham2024koszuldualitygeneralizedsteinberg}.

    Let $\lambda_{St}:W_F\rightarrow G^\vee$ be the infinitesimal character 
of $St$, by definition  $\lambda_{St}(w)=\phi_{St}(w,\begin{pmatrix}
        |w|^\frac{1}{2} & 0  \\
        0 & |w|^{-\frac{1}{2}}
        \end{pmatrix})$.

    \begin{prop}[See Lemma 1.8 of 
\cite{cunningham2024koszuldualitygeneralizedsteinberg}]\label{parameter of 
steinberg}
        The parameter $\phi_{St}:W_F\times SL_2(\mathbb{C})\rightarrow G^\vee$ 
satisfies the following:
        \begin{enumerate}
            \item $\phi_{St}$ is trivial on $W_F$.
            \item $\phi_{St}$ maps $\begin{pmatrix}
        1 & x  \\
        0 & 1
        \end{pmatrix}$ to the open $G^\vee$ orbit of $\mathcal{U}$.
        \item The infinitesimal character $\lambda_{St}$ is the same as the 
infinitesimal character of the trivial representation.
        \end{enumerate}
    \end{prop}

    Corollary \ref{main} gives a characterization of the symmetric spaces for which 
    $St_{\chi_0}$ is distinguished. To prove Theorem \ref{theorem dual} we need to
     check that the conditions given by Corollary \ref{main} are equivalent to 
$\phi_{St}$ factoring through $G_X^\vee$.

    \begin{prop}\label{quasi-split dual}
        Assume that $\iota(\mathcal{U}_X)$ intersects the open orbit in $\mathcal{U}$, then $X$ 
is quasi-split.
    \end{prop}

    \begin{proof}
        let $t\in SL_2(\mathbb{C})$ be semisimple, $\iota(t,1)\in G^\vee$ is semisimple 
and commutes with $\iota(\mathcal{U}_X)$, in particular it commutes with some 
 element in the open orbit of $\mathcal{U}$. Thus 
$\iota((t,1))$ must be in the center of $G^\vee$. By Proposition \ref{iota 
defined} it means that $\Delta_0=\emptyset$. This is 
equivalent to $X$ being quasi-split as it implies that for the Borel subgroup 
corresponding to the simple roots $\Delta$, $B\cap \sigma(B)$ is a 
torus.
    \end{proof}

    Let $\mathbf{Z}=Z(\mathbf{G})$ be the center of $\mathbf{G}$.  
$\mathbf{G}/\mathbf{Z}$ can be written as a product of adjoint simple groups 
$\mathbf{G}/\mathbf{Z}=G_1\times...\times G_m$. 

    We need to show that the opposite direction of Proposition \ref{quasi-split 
dual} holds, except in the case, where there is a $\sigma$ invariant $G_i$ 
such that $G_i\cong PGL_{2n+1}$ and $G^\sigma_i\cong P(GL_n\times GL_{n+1})$.

    \begin{prop}\label{unipotent description}
        The set $\iota(\mathcal{U}_X)$ intersects the open orbit of $\mathcal{U}$ if 
and only if the following conditions hold:
    \begin{enumerate}
        \item $X=G/H$ is quasi-split.
        \item There is no $\sigma$ invariant $G_i$, such that $G_i\cong 
PGL_{2n+1}$ and $G^\sigma_i\cong P(GL_n\times GL_{n+1})$.
     \end{enumerate}
     \end{prop}

    \begin{proof}
        We already showed that $X$ being quasi-split is a necessary condition. 
In particular $\Delta_0=\emptyset$.

        Now we assume conditions 1,2 and show that 
$\iota(\mathcal{U}_X)$ intersects the open orbit of $\mathcal{U}$.
        
        Let $\alpha\in\Delta$ be some simple root, its diagram (see Definition 
\ref{diagram}) can not contain black vertices as $\Delta_0=\emptyset$, so it 
must be one of the diagrams of Table \ref{T:table}.

Let $\mathcal{N}_X$ be the nilpotent cone of $G^\vee_X$. 
For every $\alpha\in\Delta$ we have a root space $x_\alpha$
 corresponding to $\alpha$ in the Lie algebra of $G^\vee$. 
 Let $n_\alpha:\mathbb{C}\rightarrow x_\alpha$ be an isomorphism. 
        
        The set $\iota(\mathcal{U}_X)$ intersects the open orbit of $\mathcal{U}$ if 
and only if on the level of Lie algebras $\mathcal{N}_X$ contains an element of 
the form $\sum_{\alpha\in\Delta}n_{\alpha}(a_\alpha)$ for some $a_\alpha\neq 0$.

It is enough to show that for every $\alpha\in\Delta$ there is an element in
$\mathcal{N}_X$ with a non-zero contribution from $x_\alpha$.

        We go over the three cases in Table \ref{T:table}.

        \begin{enumerate}
            \item In the first case $\alpha$ is of type 3 and 
$\alpha_1=\alpha$, $\alpha_2=-\sigma(\alpha)$. Both $\alpha$ and $-\sigma(\alpha)$
are positive simple roots. By Proposition \ref{roots of 
G_X} we have $\{n_{\alpha}(x)+n_{-\sigma(\alpha)}(x)|x\in\mathbb{C}\}\subset 
\mathcal{N}_X$.
            \item In the second case $\alpha$ is of type 2 and 
$\sigma(\alpha)=-\alpha$. By Proposition \ref{roots of G_X} we have 
$\{n_{\alpha}(x)|x\in\mathbb{C}\}\subset \mathcal{N}_X$.
            \item In the third and last case, $\alpha$ is of type 2, 
$\sigma(\alpha)\neq -\alpha$ and $\alpha, \sigma(\alpha)$ are not perpendicular. 
From the classification of involutions (see for example 
\cite{Springer1987TheCO}), this contradicts assumption number 2.
Therefore this case can not happen.
        \end{enumerate}

        For the opposite condition, it is enough to check that in the case of 
$G=PGL_{2n+1}$ and $H=P(GL_n\times GL_{n+1})$, $\iota(\mathcal{U}_X)$ does not 
intersect the open orbit of $\mathcal{U}$. In this case $G^\vee=SL_{2n+1}$ and 
$G^\vee_X=Sp_{2n}$. The result is evident.
    \end{proof}

    Now we can deduce Theorem \ref{theorem dual}.

    \begin{proof}[Proof of Theorem \ref{theorem dual}]

    For the first direction, assume that $\phi_{St}$ factors through $\iota$, 
then by Proposition \ref{parameter of steinberg}, $\iota(\mathcal{U}_X)$ 
intersects the open $G^\vee$ orbit of $\mathcal{U}$. By Proposition 
\ref{unipotent description} and Corollary \ref{main}, $St_{\chi_0}$ 
is $H$ distinguished.

    For the second direction, we assume that $St_{\chi_0}$ is $H$ distinguished 
and we need to show that $\phi_{St}$ factors through $\iota$. By Proposition 
\ref{parameter of steinberg} $\phi_{St}$ is trivial on $W_F$ so we have a map
 from $SL_2(\mathbb{C})$, 
$\phi_{St}:SL_2(\mathbb{C})\rightarrow G^\vee$ and we need to show that 
it factors through $\iota$. 
    
    The parameter $\phi_{St}$ is defined only up to conjugation.
    Fix $x\in \mathbb{C}$, by Corollary \ref{main} and Proposition \ref{unipotent description} 
we can find a conjugate of $\phi_{St}$ and $u\in \mathcal{U}_X$ such that 
$\phi_{St}(\begin{pmatrix}
        1 & x  \\
        0 & 1
        \end{pmatrix})=\iota(u)$. 
        
    The infinitesimal character of $St$ is equal to the infinitesimal character of the trivial 
representation. By Theorem 8.2 of \cite{takeda} the Langlands parameter of the 
trivial representation factors through $\iota$. Furthermore, we can find $t\in 
G^\vee_X$ such that $\phi_{St}(\begin{pmatrix}
        x^\frac{1}{2} & 0  \\
        0 & x^{-\frac{1}{2}}
        \end{pmatrix})=\iota(t)\in T$.  

    Using the Jacobson-Morozov theorem (see for example Section 3 of \cite{JMparameters}) 
we can extend the map sending $\begin{pmatrix}
        1 & x  \\
        0 & 1
        \end{pmatrix}$ to $u$ and $\begin{pmatrix}
        x^\frac{1}{2} & 0  \\
        0 & x^{-\frac{1}{2}}
        \end{pmatrix}$ to $t$, and get a map $SL_2(\mathbb{C})\rightarrow G_X^\vee$ 
        as desired. 
    \end{proof}

    \end{subsection}

\end{section}

\bibliographystyle{alpha}
\bibliography{mybib}

\end{document}